\documentclass[11pt,oneside]{amsart}
\usepackage{amsmath}
\usepackage{amsfonts}
\usepackage{amssymb}
\usepackage{amsthm}
\usepackage[colorlinks]{hyperref} 
\usepackage{wasysym}
\usepackage{graphicx}
\usepackage{dashrule}
\usepackage[all]{xy}
\usepackage{multirow}
\usepackage{enumitem}
\usepackage{tikz}
\usetikzlibrary{arrows,decorations.pathmorphing,backgrounds,positioning,fit,calc}
\usepackage{tikz-cd}
\usepackage{framed}

\normalfont\upshape

\usepackage{amscd,graphics}

\newcommand{\Proj}{{\rm Proj}}
\newcommand{\Spec}{{\rm Spec}}

\newcommand{\nc}{\newcommand}
\nc{\bla}{\phantom{bbbbb}}

\newcommand{\Hom}{ \,{\rm Hom} \,}
\newcommand{\Sym}{ \,{\rm Sym} \,}

\newcommand{\beq}{\begin{equation}}
\newcommand{\eeq}{\end{equation}}
\newcommand{\barr}{\begin{array}}
\newcommand{\earr}{\end{array}}
\newcommand{\beqar}{\begin{eqnarray}}
\newcommand{\eeqar}{\end{eqnarray}}
\newtheorem{theorem}{Theorem}[section]
\newtheorem{corollary}[theorem]{Corollary}
\newtheorem{lemma}[theorem]{Lemma}
\newtheorem{prop}[theorem]{Proposition}
\newtheorem{definition}[theorem]{Definition}
\newtheorem{remark}[theorem]{Remark}

\newtheorem{exit}[theorem]{Example}

\newenvironment{rem}{\begin{remark}\rm}{\end{remark}}

\newenvironment{defn}{\begin{definition}\rm}{\end{definition}}

\newcommand{\CC}{{\mathbb C }}
\nc{\FF}{ {\mathbb F} }
\nc{\HH}{ {\mathbb H} }
\newcommand{\ZZ}{{\mathbb Z }}
\newcommand{\PP}{ {\mathbb P } }

\newcommand{\UU}{{\mathbb U }}
\newcommand{\GG}{{\mathbb G }}

\newcommand{\cale}{\mathcal{E}}

\newcommand{\caln}{\mathcal{N}}
\newcommand{\calo}{\mathcal{O}}

\newcommand{\env}{\! 
\mathbin{\text{\rotatebox[origin=c]{70}{\scalebox{1.2}{$\approx$}}}} \! \!}

\newcommand{\inenv}{ / \! \!
\mathbin{\text{\rotatebox[origin=c]{70}{\scalebox{1.2}{$\sim$}}}} \!  }     

\nc{\conv}{{\rm Conv}}
\nc{\umax}{{U_{\max}}}
\newcommand{\weight}{\omega}

\newcommand{\Lie}{\mathrm{Lie}}

\newcommand{\bu}{\mathbf{u}}

\newcommand{\zdis}{\mathfrak{p}}

\newcommand{\OO}{\mathcal{O}}
\newcommand{\ba}{\mathbf{\alpha}}

\newcommand{\bars}{\overline{s}}
\newcommand{\barss}{\overline{ss}}

\newcommand{\lieu}{{\mathfrak u}}
\newcommand{\hU}{\hat{U}}

\newcommand{\xg}{X/\!/G}
\newcommand{\xu}{X\env U}

\newcommand{\GL}{\mathrm{GL}}
\newcommand{\SL}{\mathrm{SL}}

\newcommand{\sym}{\mathrm{Sym}}

\nc{\lieq}{{\mathfrak q}}
\nc{\liez}{{\mathfrak z}}
\nc{\lieqs}{{\lieq}^*}
\nc{\lieg}{{\mathfrak g}}
\nc{\liegs}{{\lieg}^*}
\nc{\liep}{{\mathfrak p}}
\nc{\lieps}{{\liep}^*}

\def\a{\alpha}

\def\l{\lambda}

\def\x{\xi}

\DeclareMathOperator{\spec}{Spec}

\newcommand{\nocontentsline}[3]{}
\newcommand{\tocless}[2]{\bgroup\let\addcontentsline=\nocontentsline#1{#2}\egroup}

\newcommand{\act}{\curvearrowright}

\newcommand{\dblslash}{/\! \!/}

\newcommand{\ten}{\otimes}
\newcommand{\mc}{\mathcal}

\newcommand{\mb}{\mathbb}

\newcommand{\kk}{\CC}

\newcommand{\rms}{\mathrm{s}}

\title{Geometric invariant theory for graded unipotent groups and applications}

\author{Gergely B\'erczi, Brent Doran, Thomas Hawes, Frances Kirwan}

\thanks{Early work on this project was supported by the Engineering and Physical Sciences 
Research Council [grant numbers   GR/T016170/1,EP/G000174/1].  Brent Doran was partially supported by Swiss National Science Foundation Award 200021-138071 and the European Union’s Horizon 2020 research and innovation programme under the Marie Sklodowska-Curie grant agreement AGAGAP No 742052.}

\begin{document}

\begin{abstract}
Let $U$ be a graded unipotent group over the complex numbers, in the sense that it has an extension $\hat{U}$ by the multiplicative group such that the action of the multiplicative group by conjugation on the Lie algebra of $U$ has all its weights strictly positive. Given any action of $U$ on a projective variety $X$ extending to an action of $\hat{U}$ which is linear with respect to an ample line bundle on $X$, then provided that one is willing to replace the line bundle with a tensor power and to twist the linearisation of the action of $\hat{U}$ by a suitable (rational) character, and provided an additional condition is satisfied which is the analogue of the condition in classical GIT that there should be no strictly semistable points for the action, we show that the $\hat{U}$-invariants form a finitely generated graded algebra; moreover the natural morphism from the semistable subset of $X$ to the enveloping quotient is surjective and expresses the enveloping quotient as a geometric quotient of the semistable subset. Applying this result with $X$ replaced by its product with the projective line gives us a projective variety which is a geometric quotient by $\hat{U}$ of an invariant open subset of the product of $X$ with the affine line and contains as an open subset a geometric quotient of a U-invariant open subset of $X$ by the action of $U$. Furthermore these open subsets of $X$ and its product with the affine line can be described using criteria similar to the Hilbert-Mumford criteria in classical GIT.
\end{abstract}

\maketitle

Mumford's geometric invariant theory (GIT) allows us to construct and study
quotients of 
algebraic varieties by 
linear actions of reductive groups \cite{GIT,New,Popov,PopVin,Ress}. 
When a complex reductive group $G$ acts   linearly (with respect to an ample line bundle $L$) on a complex projective variety  $X$, the associated GIT quotient
$X/\!/G$ is the projective variety $\Proj(\bigoplus_{k \geq 0} H^{0}(X,L^{\otimes k})^G)$ associated to the ring of invariants
 $\bigoplus_{k \geq 0} H^{0}(X,L^{\otimes k})^G$, which is a finitely generated
graded complex algebra. Geometrically the variety $X/\!/G$ can be described as the image
of a surjective morphism from an open subset $X^{ss}$ of $X$, consisting of the
semistable points for the action, or as $X^{ss}$ modulo the equivalence relation
$\sim$ such that if $x,y \in X^{ss}$ then $x \sim y$ if and only if the closures of the $G$-orbits of $x$ and $y$ meet in $X^{ss}$. The stable points for the action form a subset $X^s$ of $X^{ss}$ which has a geometric quotient $X^s/G$ which is an open subset of $X/\!/G$. Moreover the subsets $X^s$ and $X^{ss}$ can be described using the Hilbert--Mumford criteria for (semi)stability. The GIT quotient $X/\!/G$ and its open subset $X^s/G$ can also be described in terms of symplectic geometry
and a moment map \cite{K,Ness}.

In
suitable situations GIT can be generalised to allow us to construct
GIT-like quotients for actions of non-reductive groups
\cite{
bri09,bri11,bri15-2,brisu,
DK,dk15,F2,F1,GP1,GP2,KPEN,W}. However there is an immediate difficulty in extending GIT to linear actions of non-reductive groups, since now the ring of invariants is not necessarily finitely generated as a graded algebra, and when it is not finitely generated there is no associated projective variety. 

 Every affine algebraic group $H$ has a unipotent radical $U \unlhd
 H$ such that $R=H/U$ is reductive
(over $\CC$ we have a semi-direct product decomposition
$H \cong R \ltimes U$), and understanding GIT-theoretic questions about the action -- such as whether invariants are finitely generated -- often follows from  understanding the action of the unipotent group $U$. In some cases the $U$-invariants happen to be finitely generated. For example, if $U$ is
the unipotent radical of a parabolic subgroup $P$ of a complex reductive group
$G$ and an action of $U$  on a complex projective variety  $X$,
 which is linear with respect to an ample line bundle $L$,
extends to a linear action of $G$, then
the ring of invariants $\bigoplus_{k \geq 0} H^{0}(X,L^{\otimes k})^U$ is finitely generated \cite{Grosshans,implone}. In this case the
\lq enveloping quotient' $\xu$ (in the sense of \cite{DK} but using the notation of \cite{BDHK}) is the projective variety
$\Proj(\bigoplus_{k \geq 0} H^{0}(X,L^{\otimes k})^U)$ associated to the ring of invariants, and it contains as an open subset a geometric quotient $X^s/U$
where $X^s$ is a $U$-invariant open subset of $X$. However there is still no analogue for $\xu$ of the geometric
description of $\xg$ when $G$ is reductive as $X^{ss}$ modulo an equivalence relation, since the natural morphism from $X^{ss}$ to $\xu$ is not in general surjective, although there are alternative geometric descriptions \cite{implone}.

In this paper we consider a more general situation.
Instead of taking $U$ to be the unipotent radical of a parabolic subgroup
of a complex reductive group $G$ which acts linearly on $X$, we assume that $U$ is a unipotent
group over $\CC$ with an extension $\hat{U} = U \rtimes \CC^*$ by $\CC^*$
such that the action of $\CC^*$ by conjugation on the Lie algebra of $U$ has all its weights strictly positive; we call such a $U$ a graded unipotent group. (The unipotent radical of a parabolic subgroup
of a complex reductive group $G$ always has such an extension contained in the parabolic subgroup).
We are interested in linear actions of $U$ on projective varieties $X$ which extend to
 linear actions of $\hat{U}$. 
Given any action of $U$ on a projective variety $X$ extending to an action of $\hU$ which is linear with respect to an ample line bundle on $X$, then {\it provided} that we are willing to replace the line bundle with a tensor power and to twist the linearisation of the action of $\hU$ by a suitable (rational) character of $\hU$, and provided an additional condition is satisfied which is the analogue of the condition in classical GIT that there should be no strictly semistable points for the action (that is,  \lq semistability coincides with stability'), we find that the $\hU$-invariants form a finitely generated algebra; moreover the natural morphism $\phi: X^{ss,\hU} \to X\env \hU$ is surjective and indeed expresses $X\env \hU$ as a geometric quotient of $X^{ss,\hU}$, so that $\phi$  satisfies $\phi(x) = \phi(y)$ if and
only if the $\hU$-orbits of $x$ and $y$ coincide in $X^{ss,\hU}$. 
Applying this result with $X$ replaced by $X \times \PP^1$ gives us a projective variety $(X \times \PP^1) \env \hU$ which is a geometric quotient by $\hat{U}$ of a $\hat{U}$-invariant open subset of $X \times \CC$ and contains as an open subset a geometric quotient of a $U$-invariant open subset $X^{\hat{s},U}$ of $X$ by $U$. Furthermore the subsets $X^{s,\hU} = X^{ss,\hU}$ and $X^{\hat{s},U}$ of $X$ can be described using Hilbert--Mumford-like criteria.

This situation arises even for the Nagata counterexamples to Hilbert's 14th problem, which provide examples of linear actions of unipotent groups $U$ on  projective space such that the corresponding $U$-invariants are not finitely generated. In these cases the linear action extends to a linear action of an extension $\hat{U} = U \rtimes \CC^*$ by $\CC^*$
such that the action of $\CC^*$ by conjugation on the Lie algebra of $U$ has all its weights strictly positive. Thus when the condition that semistability coincides with stability is satisfied, we obtain open subsets $X^{s,\hU} = X^{ss,\hU}$ and $X^{\hat{s},U}$ of $X$, which are determined by analogues of the Hilbert--Mumford criteria, with geometric quotients $X^{s,\hU}/\hU$ and $X^{\hat{s},U}/U$, such that $X^{s,\hU}/\hU$ is projective and $X^{\hat{s},U}/U$ is quasi-projective with a projective completion in which the complement of  $X^{\hat{s},U}/U$ is $X^{s,\hU}/\hU$.

A related situation is studied in \cite{BK15}, where it is assumed that the linear action of the  graded unipotent group $U$ extends to a linear action of a general linear group $\GL(n)$. Here $U$ and $\hat{U}$ are embedded in $\GL(n)$ as subgroups \lq generated along the first row' in the
sense that there are integers $1 = \weight_1 < \weight_2 \leq \weight_3 \leq \cdots \leq \weight_n$
and polynomials $\zdis_{i,j}(\alpha_1,\ldots,\alpha_n)$ in $\alpha_1,\ldots,\alpha_n$
with complex coefficients for $1<i<j \leq n$ such that 
\begin{equation} \label{label1}
\hU=\left\{\left(\begin{array}{ccccc}\a_1 & \a_2 & \a_3 & \ldots & \a_n \\ 0 & \a_1^{\weight_2} & p_{2,3}(\ba) & \ldots & p_{2,n}(\ba) \\ 0 & 0 & \a_1^{\weight_3} & \ldots & p_{3,n}(\ba) \\ \cdot & \cdot & \cdot & \cdot &\cdot \\ 0 & 0 & 0 & 0 & \a_1^{\weight_n}  \end{array}\right)
: \ba =(\a_1,\ldots, \a_n) \in \CC^* \times \CC^{n-1} \right\}
\end{equation}
and $U$ is the unipotent radical of $\hU$, defined by $\alpha_1 = 1$.
The main results of \cite{BK15} also involve the subgroup $\tilde{U}$ of $\SL(n)$ which is the intersection of $\SL(n)$ with the product $\hat{U} Z(\GL(n))$ of $\hat{H}$ with the central one-parameter subgroup $Z(\GL(n)) \cong \CC^*$ of $\GL(n)$. Like $\hU$, the subgroup $\tilde{U}$ of $\GL(n)$ is a semi-direct product
$\tilde{U} = U \rtimes \CC^*$
where $\CC^*$ acts on the Lie algebra of $U$ with all weights strictly positive.
When $\GL(n)$ acts linearly on a projective variety $X$ with respect to an ample line bundle $L$ on $X$, and the linearisation of the action of $\tilde{U}$ on $X$ is twisted by a suitable rational character $\chi$ (which is \lq well adapted' to the action in the sense of \cite{BK15}), then it is shown in \cite{BK15} Theorem 1.1  that the corresponding algebra of ${\tilde{U}}$-invariants is 
finitely generated, and the projective variety $X \env \tilde{U}$ associated to this algebra of invariants is a categorical quotient of an open subset $X^{ss,\tilde{U}}$ of $X$ by $\tilde{U}$ and contains as an open subset a geometric quotient of an open subset $X^{s,\tilde{U}}$ of $X$.  
Applying a similar argument after replacing $X$ with $X \times \PP^1$ provides a projective variety $(X \times \PP^1) \env \tilde{U}$ which is a categorical quotient by $\tilde{U}$ of a $\tilde{U}$-invariant open subset of $X \times \CC$ and contains as an open subset a geometric quotient of a $U$-invariant open subset $X^{\hat{s},U}$ of $X$ by $U$.

The results of this paper are more general than those of \cite{BK15} in that the linear action of the unipotent group $U$ is only required to extend to a linear action of $\hU$ rather than to a general linear group in which $U$ and $\hU$ are embedded in a very special way. On the other hand in \cite{BK15} the additional condition that \lq semistability coincides with stability' is not required. The removal of this additional condition is addressed in \cite{BDHK16}, using a partial desingularisation construction analogous to that of \cite{K2}.

Let $\chi: \hU \to \CC^*$ be a character of $\hat{U}$ with kernel containing $U$; we will identify such characters $\chi $ with integers so that the integer 1 corresponds to the character which fits into the exact sequence $U \to \hU \to \CC^*$. Suppose that $ \weight_{\min} = \weight_0 < \weight_{1} < 
\cdots < \weight_{h} = \weight_{\max} $ are the weights with which the one-parameter subgroup $\CC^* \leq \hU$ acts on the fibres of the tautological line bundle $\calo_{\PP((H^0(X,L)^*)}(-1)$ over points of the connected components of the fixed point set $\PP((H^0(X,L)^*)^{\CC^*}$ for the action of $\CC^*$ on $\PP((H^0(X,L)^*)$; when $L$ is very ample $X$ embeds in $\PP((H^0(X,L)^*)$ and the line bundle $L$ extends to the dual $\calo_{\PP((H^0(X,L)^*)}(1)$ of the tautological line bundle $\calo_{\PP((H^0(X,L)^*)}(-1)$.
We will assume that there exist at least two distinct such weights since otherwise the action of $U$ on $X$ is trivial.
Let $c$ be a positive integer such that 
$$  \weight_{\min} = \weight_{0} <
\frac{\chi}{c} < \weight_{1} ;$$
 we will call rational characters $\chi/c$  with this property {\it  adapted } to the linear action of $\hU$ on $L$, and we will call the linearisation adapted if $ \weight_{0} <0 < \weight_{1} $.
 The linearisation of the action of $\hat{U}$ on $X$ with respect to the ample line bundle $L^{\otimes c}$ can be twisted by the character $\chi$ so that the weights $\weight_j$ are replaced with $\weight_jc-\chi$;
let $L_\chi^{\otimes c}$ denote this twisted linearisation. 
Let $X^{s,\CC^*}_{\min+}$ denote the stable subset of $X$ for the linear action of $\CC^*$ with respect to the linearisation $L_\chi^{\otimes c}$; 
then
$$X^{s,\CC^*}_{\min+} = X^0_{\min} \setminus Z_{\min}
$$
where $Z_{\min}$ is the union of the connected components of the fixed point set $X^{\CC^*}$ for the action of $\CC^*$ on $X$ given by
$$Z_{\min} = \{ x \in X^{\CC^*} \, | \, \CC^* \mbox{ acts on $L^*|_x$ with weight }  \weight_{\min} \} $$
and
$$ X^0_{\min} = \{ x \in X \, | \, \lim_{t \to 0} \, tx \in Z_{\min} \}
.$$
We set $$X^{s,\hU}_{\min+} = X \setminus \hU (X \setminus X^{s,\CC^*}_{\min+}) = \bigcap_{u \in U} u X^{s,\CC^*}_{\min+}$$ to be the complement of the $\hU$-sweep (or equivalently the $U$-sweep) of $X \setminus X^{s,\CC^*}_{\min+}$.

In order to state the main theorem of this paper it is necessary to strengthen the condition that the linearisation is adapted. We will say that a property holds for a linear action of $\hat{U}$ with respect to a linearisation twisted by a {\it well adapted} rational character if there exists $\epsilon >0$ such that if $\chi/c$ is any rational character of $\CC^*$ (lifted to $\hU$) with 
$$  \weight_{\min} <
\frac{\chi}{c} < \weight_{\min} + \epsilon$$
then the property holds for the induced linearisation on $L^{\otimes c}$ twisted by $\chi$.

We also require that the action of $U$ satisfies an additional condition to which we will refer as  \lq semistability coincides with stability'. More precisely,
 we require that whenever $U'$ is a subgroup of $U$ normalised by $\CC^*$ and $\xi$ belongs to the Lie algebra of $U$ but not the Lie algebra of $U'$, then the weight space with weight $- \weight_{\min}$ for the action of $\CC^*$ on $H^0(X,L)$ is contained in the image $\delta_\xi(H^0(X,L)^{U'})$ of $H^0(X,L)^{U'}$ under the 
infinitesimal action  $\delta_\xi: H^0(X, L) \to H^0(X, L)$ of $\xi$ on $H^0(X,L)$.

\begin{rem} For motivation for the terminology \lq semistability coincides with stability' see Remark \ref{remsss}, and also \cite{BDHK16} where a slightly weaker interpretation of this terminology is used.
\end{rem}

\begin{theorem} \label{mainthm} 
Let  $\hat{U} = U \rtimes \CC^*$ be a semidirect product of the unipotent group $U$ by $\CC^*$,
where the conjugation action of $\CC^*$ on $U$ is such that all the weights
of the induced $\CC^*$-action on the Lie algebra of $U$ are strictly positive.
Suppose that $\hU$  acts linearly on a projective variety $X$ with respect to a very ample line bundle $L$, and that  this linear action  satisfies the condition that \lq semistability coincides with stability' as above. Suppose also that 
$\chi: \hU \to \hU/U \to \CC^*$ is a character of $\hat{U}$ and $c$ is a positive integer such that the rational character $\chi/c$ is  adapted to the linear action of $\hU$ on $X$.
Then \\
(i) $X^{s,\hU}_{\min+}$ is a $\hU$-invariant open subvariety of $X$ with a geometric quotient $X^{s,\hU}_{\min+}/\hat{U}$  by $\hat{U}$ which is a projective variety.

If moreover the rational character $\chi/c$ is well adapted, then \\ (ii) 
the algebra of invariants $\oplus_{m=0}^\infty H^0(X,L_{m\chi}^{\otimes cm})^{\hat{U}}$ is 
finitely generated, $X^{s,\hU}_{\min+}$ is the corresponding (semi)stable locus for the linear action of $\hU$ on $X$ and 
the enveloping quotient $$X\env \hat{U} \cong X^{s,\hU}_{\min+}/\hat{U}$$ is the projective variety associated to this algebra of invariants.  
\end{theorem}

\begin{rem}
Recall that Popov \cite{Popov} has shown that if $H$ is {\it any} non-reductive group then there is an affine variety $Y$ on which $H$ acts such that the algebra of invariants
$\calo(Y)^H$ is {\it not} finitely generated.
\end{rem}

Applying Theorem \ref{mainthm}  after replacing $X$ with $X \times \PP^1$ we  obtain geometric information about the action of the unipotent group $U$ on $X$: 

\begin{corollary} \label{cor:invariants}
In the situation above let $\hat{U}$ act diagonally on $X \times \PP^1$  where the action on $\PP^1$ is via \[\hat{u} \cdot [x:y]=[\chi_1 (\hat{u})x:y]\]
where $\chi_1:\hU \to \CC^*$ is the character of $\hU$ with kernel $U$ which fits into the extension $ \{1\} \to U \to \hU \to \CC^* \to \{ 1\} $,
and linearise this action using the tensor product of $L_\chi$ with $\calo_{\PP^1}(M)$ for suitable $M \geq 1$. Then $(X \times \PP^1) \env \hat{U}$ is a projective variety which is a geometric quotient by $\hat{U}$ of a $\hat{U}$-invariant open subset of $X \times \CC$ and contains as an open subset a geometric quotient of a $U$-invariant open subset $X^{\hat{s},U}$ of $X$ by $U$. 
\end{corollary}

\begin{rem}\label{cor:invariants2}
We can also deduce that the algebra  $A = \oplus_{m=0}^\infty H^0(X\times \PP^1,L_{m\chi}^{\otimes cm} \otimes \mathcal{O}_{\PP^1}(M))^{\hat{U}}$ of $\hat{U}$-invariants on $X \times \PP^1$  is 
finitely generated for a well-adapted rational character $\chi/c$ of $\hat{U}$ when $c$ is a sufficiently divisible positive integer. This graded algebra $A$ can be identified with the subalgebra of the algebra of $U$-invariants  $\oplus_{m=0}^\infty H^0(X,L^{\otimes cm})^{{U}}$ on $X$ generated by the $U$-invariants in $\oplus_{m=0}^\infty H^0(X,L^{\otimes cm})^{{U}}$ which are weight vectors with non-positive weights for the action of $\CC^* \leq \hat{U}$ after twisting by the well-adapted rational character $\chi/c$.  The sections $\sigma$ of $L$ which are weight vectors with weight $-\weight_{min}$ are all $U$-invariant, and after twisting by $\chi/c$  these are the only weight vectors in $H^0(X,L)$ which have non-positive (in fact strictly negative) weights. If we localise the $U$-invariants at any such $\sigma$ then we get a finitely generated algebra of invariants $\OO(X_\sigma)^U$, since this algebra can be identified with the localisation of $A$ at $\sigma$. This can be proved directly, and leads to an alternative proof of Theorem \ref{mainthm} (cf. \cite{BDHK16}.
\end{rem}

Theorem \ref{mainthm} has another immediate corollary:

\begin{corollary} \label{cor1.2}
Let $H \cong R \ltimes U$ be a complex linear algebraic group with unipotent radical
$U$ and $R \cong H/U$ reductive, and suppose that $R$ contains a central
subgroup isomorphic to $\CC^*$ which acts by conjugation on the Lie algebra of $U$ with all weights strictly positive. Let $\hU$ be the subgroup of $H$ which is the semidirect product of $U$ and this central one-parameter subgroup $\CC^*$ of $R$.  Suppose that  $H$ acts linearly on a projective
variety $X$ 
 with respect to an ample line bundle $L$, and that $\chi: H \to \CC^*$ is a character of $H$, that $c$ is a sufficiently divisible positive integer such that the restriction to $\hU$ of the rational character $\chi/c$ is well adapted for the linear action of $\hU$ on $X$, and that the linear action of the unipotent radical $U$ satisfies the condition that \lq semistability coincides with stability' as above. Then the algebra of $H$-invariants $\oplus_{m=0}^\infty H^0(X,L_{m\chi}^{\otimes cm})^{H}$ is 
finitely generated, and the projective variety $X\env  H$ associated to this algebra of invariants is a categorical quotient of an open subvariety $X^{ss,H}$ of $X$ by $H$, and the canonical $H$-invariant morphism $\phi: X^{ss,H} \to X \env H$ is surjective with $\phi(x) = \phi(y)$ if and only if the closures of the $H$-orbits of $x$ and $y$ meet in $X^{ss,H}$.  
\end{corollary}

\begin{proof}
This result follows by quotienting $X$ first by $\hU$, using Theorem \ref{mainthm}, and then by the induced linear action of the reductive group $H/\hU \cong R/\CC^*$.
\end{proof}

\begin{rem} \label{wellvsadapted}
Note that we require the  rational character $\chi/c$ to be  {\it well} adapted to ensure that the algebra of $H$-invariants $\oplus_{m=0}^\infty H^0(X,L_{m\chi}^{\otimes cm})^{H}$ is 
finitely generated and that we obtain an induced ample line bundle on $X\env  H$. However, in order to obtain the  open subvariety $X^{ss,H}$ of $X$ and its categorical quotient which is the projective variety $X\env  H$, it is enough for the rational character to be merely adapted, not well adapted. Moreover both  $X^{ss,H}$ and the stable locus $X^{s,H}$ (which has a geometric quotient by the $H$-action) are determined by Hilbert--Mumford criteria, exactly as for classical GIT.
\end{rem}

\noindent {\bf Example:} 
Consider the weighted projective plane $\PP(1,1,2)$ which is $\CC^3 \setminus\{0\}$ modulo the action of $\CC^*$ with weights $1,1,2$.
The automorphism group of $\PP(1,1,2)$ is 
$$\mbox{Aut}(\PP(1,1,2)) \cong R \ltimes U$$
with $R \cong GL(2)$ reductive and 
$U \cong (\CC^+)^3$ unipotent; here $(\lambda,\mu,\nu) \in (\CC^+)^3 \cong U$ acts on 
 the weighted projective plane $\PP(1,1,2)$
as $[x,y,z]  \mapsto [x,y,z+\lambda x^2 + \mu xy + \nu y^2]$.
The central one-parameter subgroup $\CC^*$ of $R \cong GL(2)$ acts on
$Lie(U)$ with all positive weights, and the
associated extension
$\hat{U} = U \rtimes \CC^*$
can be identified with a subgroup of $\mbox{Aut}(\PP(1,1,2))$.
Thus Corollary \ref{cor1.2} applies to every linear action 
of $\mbox{Aut}(\PP(1,1,2))$ on a projective variety $X$ with respect to an ample line bundle $L$ after twisting by a well adapted rational character. 

The weighted projective plane $\PP(1,1,2)$ is a simple example of a toric
variety; in fact as we shall see in $\S$4 below, the automorphism group of
any complete simplicial toric variety satisfies the conditions of 
 Corollary \ref{cor1.2}.

\medskip

Our first motivation for considering linear actions of groups of the form
$\hat{U}$ in this article and in \cite{BK15} came from the study of jet differentials. 
The groups $\GG_k$ of $k$-jets of holomorphic reparametrizations of $(\CC,0)$ (and more generally the groups $\GG_{k,p}$ of $k$-jets of holomorphic reparametrizations of $(\CC^p,0)$ for $p \geq 1$) play an important role in the strategy of Demailly, Siu and others \cite{ahl,blo,dem,dmr,dr,gg, kob,merker, siu1,siu2,siu3} towards the Green-Griffiths conjecture on entire holomorphic curves in hypersurfaces of large degree in projective spaces. Here $\GG_k$ is a non-reductive complex linear
algebraic group which is a semi-direct product $\GG_k = 
 \UU_k \rtimes \CC^*$ of its unipotent radical $\UU_k$ by $\CC^*$ acting 
with weights $1,2,3,\ldots,k$ on the Lie algebra of $\UU_k$, while if $p>1$ then $\GG_{k,p} = 
 \UU_{k,p} \rtimes GL(p;\CC)$ where all the weights of the central one-parameter subgroup $\CC^*$ of $GL(p;\CC)$ on the Lie algebra of the unipotent radical $\UU_{k,p}$ 
of $\GG_{k,p}$ are strictly positive. So the results above apply
to linear actions of the
reparametrization group $\GG_k$  and 
its generalizations $\GG_{k,p}$
for $p \geq 1$).
In particular the reparametrization group $\GG_k$ acts fibrewise in a natural way on the  Semple jet bundle $J_k(T^*X) \to X$ over a complex manifold $X$ of dimension $n$ with fibre 
$$J_{k,x} \cong \bigoplus_{j=1}^k {\rm Sym}^j(\CC^n)$$
at $x$ consisting of the $k$-jets of holomorphic curves at $x$.
There is an induced action of $\GG_k$ on the polynomial ring
$\calo(J_{k,x})$, 
which can be identified with the algebra
$\oplus_{m=0}^\infty H^0(\PP(J_{k,x}),\calo_{ \PP(J_{k,x}) }(1)^{\otimes m})$ of sections of powers of the hyperplane line bundle on the associated projective space $\PP(J_{k,x})$,
and the bundle 
$E_{k} \to X$ of Demailly-Semple invariant jet differentials
of order $k$ 
has fibre at $x$ given by
$(E_{k})_x = \calo(J_{k,x})^{\UU_k}.$



\medskip

The layout of this paper is as follows. $\S$1 reviews 
the results of \cite{DK} and \cite{BDHK} on non-reductive GIT, and 
$\S$2 considers the case when $\dim(U)=1$ and proves Theorem \ref{mainthm} in this case. $\S$3 uses these results to prove Theorem \ref{mainthm}
and Corollaries \ref{cor:invariants} and \ref{cor1.2}. 
In $\S$4 we observe that Corollary \ref{cor1.2} applies to the automorphism groups of all complete simplicial toric
varieties, while  
$\S$5 discusses applications to Demailly-Semple jet differentials and their generalisations to maps $\CC^p \to X$.  

\medskip

The authors would like to thank the anonymous referee for helpful suggestions on streamlining and clarifying our arguments.

\section{Classical and non-reductive geometric invariant theory}

Let $X$ be a complex quasi-projective variety and let $G$ be a complex
reductive group acting on $X$. 
 To apply (classical) geometric invariant theory (GIT)
we require a  {linearisation} of the action; that is, a
line bundle $L$ on $X$ and a lift $\mathcal{L}$ of the action of $G$ to $L$.

\begin{rem} Usually $L$ is assumed to be ample, and it makes no difference for classical GIT 
if we replace $L$ with $L^{\otimes k}$ for any integer $k>0$,
so then we lose little generality in supposing that for some
projective embedding 
$X \subseteq \PP^n$ 
the action
of $G$ on $X$ extends to an action on $\PP^n$ given by a
representation 
$$\rho:G\rightarrow GL(n+1),$$
and taking for $L$ the hyperplane line bundle on $\PP^n$.
\end{rem}

\begin{defn} \label{defn:s/ssred}
Let $X$ be a { quasi-projective} complex variety with an action of a 
complex reductive
group $G$ and linearisation $\mathcal{L}$ with respect to a line
bundle $L$ on $X$. Then $y \in X$ is {\em
semistable} for this linear action if there exists some $m > 0$
and $f \in H^0(X, L^{\otimes m})^G$ not vanishing at $y$ such that
the open subset
$$ X_f := \{ x \in X \ | \ f(x) \neq 0 \}$$
is affine, and $y$ is {\em stable} if also the action of $G$
on $X_f$ is closed with all stabilisers finite.
\end{defn}

\begin{rem} \label{reductiveenvquot} This definition comes from \cite{GIT}, although in \cite{GIT} the terminology \lq properly stable' is used instead of stable.
When
$X$ is projective and $L$ is ample and $f \in H^0(X, L^{\otimes
m})^G$ for $m > 0$, then $X_f$ is affine
if and only if $f$ is nonzero.
The reason for introducing the requirement that $X_f$ must be affine
in Definition \ref{defn:s/ssred} above is to ensure that $X^{ss}$ has
a quasi-projective categorical quotient $X^{ss} \to \xg$, which restricts to a geometric
quotient $X^s \to X^s/G$ (see \cite{GIT} Theorem 1.10). \end{rem}

From now on in this section we will assume that $X$ is projective and $L$ is ample. 
We have an induced action of $G$ on the
homogeneous coordinate ring 
$${\hat{\calo}}_L(X) = \bigoplus_{k \geq 0} H^0(X, L^{\otimes k}) $$
of $X$.
The
subring ${\hat{\calo}}_L(X)^G$ consisting of the elements of ${\hat{\calo}}_L(X)$
left invariant by $G$ is a finitely generated graded complex algebra
because $G$ is reductive, and
 the GIT quotient $X /\!/ G$ is the projective variety  $\Proj({\hat{\calo}}_L(X)^G)$. The subsets $X^{ss}$ and $X^s$ of $X$ are characterised by the following properties (see
\cite[Chapter 2]{GIT} or \cite{New}). 

\begin{prop} (Hilbert-Mumford criteria)
\label{sss} (i) A point $x \in X$ is semistable (respectively
stable) for the action of $G$ on $X$ if and only if for every
$g\in G$ the point $gx$ is semistable (respectively
stable) for the action of a fixed maximal torus of $G$.

\noindent (ii) A point $x \in X$ with homogeneous coordinates $[x_0:\ldots:x_n]$
in some coordinate system on $\PP^n$
is semistable (respectively stable) for the action of a maximal 
torus of $G$ acting diagonally on $\PP^n$ with
weights $\a_0, \ldots, \a_n$ if and only if the convex hull
$$\conv \{\a_i :x_i \neq 0\}$$
contains $0$ (respectively contains $0$ in its interior).
\end{prop}

Now let $H$ be any affine algebraic 
group, with unipotent radical $U$, acting linearly on a complex projective variety $X$ with respect to an ample line bundle $L$. Then the ring of invariants 
$${\hat{\calo}}_L(X)^H = \bigoplus_{k \geq 0} H^0(X, L^{\otimes k})^H$$
is not necessarily finitely generated as a graded complex algebra,
so that $\Proj({\hat{\calo}}_L(X)^H)$ is not well-defined as a projective variety, although
$\Proj({\hat{\calo}}_L(X)^H)$ does make sense as a scheme, and the
inclusion of ${\hat{\calo}}_L(X)^H$ in ${\hat{\calo}}_L(X)$ gives us a rational map of schemes 
$q$ from $ X$ to $ \Proj({\hat{\calo}}_L(X)^H)$, whose
image is a constructible subset of $\Proj({\hat{\calo}}_L(X)^H)$ (that is, a finite union of
locally closed subschemes). 
The action on $X$ of the unipotent radical $U$ of $H$
is studied in \cite{DK} following earlier work \cite{F2,F1,GP1,GP2,W}.

\begin{defn} (See \cite{DK} $\S$4). \label{defnssetc}
Let $I = \bigcup_{m>0} H^0(X,L^{\otimes m})^U$
and for $f \in I$ let $X_f$ be the $U$-invariant affine open subset
of $X$ where $f$ does not vanish, with ${\calo}(X_f)$ its coordinate ring. 
A point $x \in X$ is called {\em naively semistable} if 
 there exists some $f \in I$
which does not vanish at $x$, and the set of naively semistable points 
is denoted $X^{nss}= \bigcup_{f \in I} X_f$.
The {\em finitely generated semistable set} of $X$ is  $X^{ss,
fg} =  \bigcup_{f \in I^{fg}} X_f$ where
$$I^{fg} = \{f
\in I \ | \ {\calo}(X_f)^U
\mbox{ is finitely generated }   \}.$$
The set of {\em naively stable}
points of $X$ is
     $X^{ns} = \bigcup_{f \in I^{ns}} X_f$ where
$$I^{ns} = \{f
\in I^{fg} \ | \  
  q: X_f \longrightarrow
\Spec({\calo}(X_f)^U) \mbox{ is a geometric quotient} \},$$
and the set of {\em locally trivial stable} points is $ X^{lts} =
\bigcup_{f \in I^{lts} } X_f$ where
\begin{eqnarray*} I^{lts}\ \  =\ \  \{f
\in I^{fg}  \ | \  
    q: X_f \longrightarrow \Spec({\calo}(X_f)^U) \mbox{ is a locally trivial
geometric quotient} \}. \end{eqnarray*}
\label{defn:envelopquot}
The {\em enveloped quotient} of
$X^{ss,fg}$ is $q: X^{ss, fg} \rightarrow q(X^{ss,fg})$, where
$q: X^{ss, fg} \rightarrow \Proj({\hat{\calo}}_L(X)^U)$ is the natural
morphism of schemes and
$q(X^{ss,fg})$ is a dense constructible subset of the {\em
enveloping quotient}
$$X \env  U = \bigcup_{f \in I^{fg}}
\Spec({\calo}(X_f)^U)$$ of $X^{ss, fg}$. 

Motivated by \cite{DK} 5.3.1 and 5.3.5,
we call 
a point $x \in X$ 
{\em stable} for the linear $U$-action if $x \in X^{lts}$ and {\em semistable} if $x
\in X^{ss, fg}$. We write $X^s$ (or $X^{s,U}$) for $X^{lts}$,
and we write $X^{ss}$ (or $X^{ss,U}$) for $X^{ss,fg}$ (cf. \cite{DK} 5.3.7).
\end{defn}


\begin{rem} $q(X^{ss})$ is 
not necessarily a subvariety of $X \env U $ (see for  example \cite{DK} $\S$6).
\end{rem}

\begin{rem} 
If ${\hat{\calo}}_L(X)^U$ is finitely generated then $X\env U$ is the corresponding projective
variety $\Proj({\hat{\calo}}_L(X)^U)$ \cite{DK}.
 In    \label{remclaim}
{\rm \cite{DK} 4.2.9 and 4.2.10} it is also claimed that 
 \label{patching2} 
the enveloping quotient
$X \env U$ is alwasy a quasi-projective variety with an ample line bundle
$L_H \to X \env U$ which pulls back to a positive tensor power of $L$
under the natural map $q:X^{ss} \to X \env U$. The argument given there fails in general since the morphisms $X_f \to 
\Spec({\calo}(X_f)^U)$ for $f \in I^{ss,fg}$ are not necessarily surjective. However it is still true that the enveloping quotient 
$X \env  U$ has quasi-projective open subvarieties (\lq inner enveloping quotients' $X \inenv H$) which contain the enveloped quotient $q(X^{ss})$  and have ample line bundles pulling back to  positive tensor powers of $L$
under the natural map $q:X^{ss} \to X\env U$ (see \cite{BDHK} for details). 
\end{rem}

The results of \cite{DK}  can be generalised to allow us to study $H$-invariants for linear algebraic groups $H$ which are neither unipotent nor reductive \cite{BDHK,BK15}. Over $\CC$ any linear algebraic group $H$ is a semi-direct product $H=U\rtimes R$ where $U \subset H$ is the unipotent radical of $H$ (its maximal unipotent normal subgroup) and $R\simeq H_r = H/U$ is a reductive subgroup of $H$. When $H$ acts linearly on a projective variety $X$ with respect to an ample line bundle $L$,
 the naively semistable and (finitely generated) semistable sets $X^{nss}$ and $X^{ss} = X^{ss,fg}$, enveloped and enveloping quotients and inner enveloping quotients
$$q: X^{ss} \to q(X^{ss}) \subseteq X \inenv H \subseteq X \env H$$
 are defined in \cite{BDHK}  as for the unipotent case in Definition \ref{defnssetc} and Remark \ref{remclaim}. However the definition of the stable set $X^s$ combines the unipotent and reductive cases as follows.

\begin{defn} \label{def:GiSt1.1}
Let $H$ be a linear algebraic group acting on an irreducible variety $X$ and $L \to X$ a
linearisation for the action. The \emph{stable locus} is the open subset
\[
X^{\rms}= \bigcup_{f \in I^{\rms}} X_f
\]
of $X^{ss}$, where $I^{\rms} \subseteq \bigcup_{r>0} H^0(X,L^{\ten r})^H$ is the subset
of $H$-invariant sections satisfying the following conditions:
\begin{enumerate}
\item \label{itm:GiSt1.1-1} the open set $X_f$ is affine;
\item \label{itm:GiSt1.1-2} the action of $H$ on $X_f$ is closed with all
  stabilisers finite groups; and
\item \label{itm:GiSt1.1-3} the restriction of the $U$-enveloping quotient map
 \[
q_{U}:X_f \to \spec((S^{U})_{(f)})
\]
is a principal $U$-bundle for the action of $U$ on $X_f$. 
\end{enumerate}

If it is necessary to indicate the group $H$ we will write $X^{s,H}$ and $X^{ss,H}$ for $X^s$ and $X^{ss}$.
\end{defn}
\begin{rem} \label{rmk:GiSt2}
This definition of stability extends the definition of
stability in \cite{DK} for unipotent groups, and in the case
where $H$ is reductive, then $U$ is trivial and the definition reduces to Mumford's notion of 
properly stable points in \cite{GIT}.    Note that as one would hope

(i) if $R$ is a reductive subgroup of $H$ then it follows straight from the
definition that $X^{s,R} \subseteq X^{s,H}$; 

(ii) if $N$ is a normal subgroup of $H$ such that the canonical projection $U \to
U/N_u$ splits, and if $W$ is an $H$-invariant open
subvariety of $X^{s,N}$ with a geometric quotient $W/N$  which is an
$H/N$-invariant open subvariety of $X^{s,N}/N \subseteq X\inenv N$, where
$X\inenv N$ is an inner enveloping quotient of $X$ by $N$ such that  a tensor
power $L^{\otimes m}$ of $L$ induces a very ample line bundle on $X\inenv N$ and
hence an embedding of $X \inenv N$ in the corresponding projective space with
closure $\overline{X\inenv N}$, and if $W/N \subseteq (\overline{X \inenv
  N})^{s,H/N}$, then $W \subseteq X^{s,H}$. 

Note also however that, in contrast with classical GIT for reductive group actions with ample linearisations, if $Y$ is an $H$-invariant subvariety of $X$ and we restrict the linearisation of the $H$-action on $X$ to $Y$, then $Y^{s,H}$ and $Y^{ss,H}$ do {\it not} necessarily coincide with $Y \cap X^{s,H}$ and $Y \cap X^{ss,H}$.
\end{rem}

The following result which we will need is proved in \cite{BDHK} Cor 3.1.20.

\begin{prop} \label{cor:GiFi5} Suppose $H$ is a linear algebraic group, $X$ an
  irreducible $H$-variety and $L \to X$ a linearisation. If the
  enveloping quotient $X \env H$ is quasi-compact and complete, then 
 for suitably
  divisible integers $r>0$  
the algebra of invariants $\bigoplus_{k \geq 0} H^0(X, L^{\otimes kr})^H$ is
finitely generated and the enveloping quotient $X \env H$ is the associated
projective variety; moreover the line bundle $L^{\otimes r}$ induces a very ample
line bundle $L^{\otimes r}_{[H]}$ on $X\env H$ with a natural isomorphism
\[
\bigoplus_{k \geq 0} H^0(X, L^{\otimes kr})^H \cong \bigoplus_{k \geq 0} H^0(X\env H, (L^{\otimes r}_{[H]})^{\otimes k}) .
\]
\end{prop}

 If a linear algebraic group $H = U \rtimes R$ with unipotent radical $U$ is a subgroup of a reductive group $G$ then there is an induced right action of $R$ on $G/U$ which commutes with the left action of $G$. 
Similarly if $H$ acts on a projective variety $X$ then there is an induced action of $G\times R$ on $G\times_{U}X$ with an induced $G\times R$-linearisation. The same is true if we replace the requirement that $H$ is a subgroup of $G$ with the existence of a group homomorphism $H\to G$ which is $U$-faithful (that is, its restriction to $U$ is injective).
We can thus consider the situation when $X$ is a nonsingular complex projective variety acted on by a linear algebraic group $H=U\rtimes R$, while $L$ provides  a very ample linearisation of the $H$ action defining an embedding $X\subseteq \PP^n$, and $H \to G$ is an $U$-faithful homomorphism into a reductive subgroup $G$ of $\SL(n+1;\CC)$ with respect to an ample line bundle $L$.

\begin{rem} \label{def:GiSt1.3}
\label{thm:fgcriteriageneral}
Suppose that $L'$ is a $G\times R$-linearisation over a  nonsingular projective completion $\overline{G \times_{U} X}$ of $G \times_{U} X$ extending the $G\times R$ linearisation over $G \times_{U} X$ induced by $L$. 
Let $D_1, \ldots , D_r$ be the codimension one components of the boundary of $G
\times_{U} X$ in $\overline{G \times_{U} X}$, and suppose for all sufficiently divisible $N$  that
$L'_N=L^{\prime}[N
\sum_{j=1}^r D_j]$ is an ample line bundle on $\overline{G \times_{U} X}$. 
Then as noted in \cite{BK15} the proof of \cite{DK} Theorem 5.1.18  tells us that the algebra of invariants 
$\bigoplus_{k \geq 0} H^0( X,L^{\otimes k})^H$ is finitely generated if and only if  
for all sufficiently divisible $N$ any $G\times R$-invariant section of a positive tensor power of $L'_N$ vanishes on every codimension one component $D_j$. 
Moreover when this happens the enveloping quotient satisfies
\begin{equation}\label{equotient}
X\env H =\mathrm{Proj}(\oplus_{k \geq 0} H^0( X,L^{\otimes k})^H)\cong \overline{G \times_{U} X}/\!/_{L'_N} (G\times R)
\end{equation}
for sufficiently divisible $N$. 
\end{rem}

In general even when the algebra of invariants $\bigoplus_{k \geq 0} H^0( X,L^{\otimes k})^H$ on $X$ is finitely generated and \eqref{equotient} is true, the morphism $X \to X\env H$ is not surjective and in order to study the geometry of $X\env H$ by identifying it with $\overline{G \times_{U} X}/\!/_{L'_N} (G\times R)$ we need information about the boundary $\overline{G \times_{U} X}\setminus G\times_{U} X$ of $\overline{G \times_{U} X}$. If, however, we are lucky enough to have a situation where for sufficiently divisible $N$ the line bundle $L'_N$ is  ample  and the boundary $\overline{G \times_{U} X}\setminus G\times_{U} X$ is unstable for the linear action of $G \times R$, then the picture  is almost as well behaved as for reductive group actions on projective varieties with ample linearisations, as follows.

\begin{definition} Let $X^{\barss}=X\cap \, \overline{G \times_{U} X}^{ss,G\times R}$ and $X^{\bars}=X\cap \, \overline{G \times_{U} X}^{s,G\times R}$
where $X$ is embedded in $G \times_{U} X$ in the obvious way as $x\mapsto [1,x]$. 
\end{definition}

\begin{theorem}\label{thm:geomcor} {\rm (\cite{BK15} Thm 2.9 and \cite{DK} 5.3.1 and 5.3.5))}.
Let $X$ be a complex projective variety acted on by a linear algebraic group $H=U\rtimes R$ where $U$ is the unipotent radical of $H$ and let $L$ be a very ample linearisation of the $H$ action defining an embedding $X\subseteq \PP^n$. Let $H \to G$ be an $U$-faithful homomorphism into a reductive subgroup $G$ of $\SL(n+1;\CC)$ with respect to an ample line bundle $L$. Let $L'$ be a $G\times R$-linearisation over a projective completion $\overline{G \times_{U} X}$ of $G \times_{U} X$ extending the $G\times R$ linearisation over $G \times_{U} X$ induced by $L$. 
Let $D_1, \ldots , D_r$ be the codimension $1$ components of the boundary of $G
\times_{U} X$ in $\overline{G \times_{U} X}$, and suppose  that 
$L'_N=L^{\prime}[N
\sum_{j=1}^r D_j]$ is an ample line bundle on $\overline{G \times_{U} X}$  for all sufficiently divisible $N$. If 
for all sufficiently divisible $N$ any $G\times R$-invariant section of a positive tensor power of $L'_N$ vanishes on the 
 boundary of $G
\times_{U} X$ in $\overline{G \times_{U} X}$, then 
\begin{enumerate}
\item $X^{\bar{s}} = X^s$ and $X^{ss} = X^{\overline{ss}}$;
\item the algebra of invariants 
$\bigoplus_{k \geq 0} H^0( X,L^{\otimes k})^H$ is finitely generated;
\item the enveloping quotient $X\env H \cong \overline{G \times_{U} X}/\!/_{L'_N} (G\times R)\cong \mathrm{Proj}(\oplus_{k \geq 0} H^0( X,L^{\otimes k})^H)$ for sufficiently divisible $N$;
\item $\overline{G \times_{U} X}^{ss,G\times R, L'_N} \subseteq G\times_{U} X$ and therefore the morphism 
\[ \phi: 
X^{\barss} \rightarrow X\env H\]
is surjective and $X\env H$ is a categorical quotient of $
X^{\barss}$;
\item if $x,y \in 
X^{\barss}$ then $\phi(x) = \phi(y)$ if and only if the closures of the $H$-orbits of $x$ and $y$ meet in $
X^{\barss}$;
\item $\phi$ restricts to a geometric quotient $X^{\bars} \rightarrow X^{\bars}/H \subseteq  X\env H$.
\end{enumerate} 
\end{theorem}

\section{Actions of $\CC^+ \rtimes \CC^*$}

We will prove Theorem \ref{mainthm} by induction on the dimension of $U$. In this section we will concentrate on the case when $\dim(U) = 1$ so that $U \cong \CC^+$. 

\begin{defn} \label{defn:ss=sdimone}
Let $X$ be a complex projective variety equipped with a linear action
(with respect to an ample line bundle $L$) of  a
semi-direct product $\hat{U} = \CC^* \ltimes U$, where $U$ is unipotent and
 the weights of
the induced $\CC^*$ action on the Lie algebra of
$U$ are all strictly positive.
When $\xi$ is an element of the Lie algebra of $U$ let  $\delta_\xi: H^0(X, L) \to H^0(X, L)$ define the infinitesimal action of $\xi$ on $H^0(X,L)$. 
We say that {\it semistability coincides with stability} for this linear action  if 
 whenever $U'$ is a subgroup of $U$ normalised by $\CC^*$ and $\xi$ belongs to the Lie algebra of $U$ but not the Lie algebra of $U'$, then the weight space with maximum weight $- \weight_{\min}$ for the action of $\CC^*$ on $H^0(X,L)$ is contained in the image $\delta_\xi(H^0(X,L)^{U'})$ of $H^0(X,L)^{U'}$ under the 
infinitesimal action  $\delta_\xi: H^0(X, L) \to H^0(X, L)$ of $\xi$ on $H^0(X,L)$.
\end{defn}

\begin{rem}  \label{remsss}
When $U=\CC^+$ has dimension one, this says that for any nonzero $\xi \in {\rm Lie}(U)$
 the weight space with maximum weight $- \weight_{\min}$ for the
action of $\CC^*$ on $H^0(X,L)$ is contained in the image
$\delta_\xi(H^0(X,L))$ of $H^0(X,L)$ under 
$\delta_\xi$. This will be the case unless $\delta_{\xi}$ has a Jordan block of size 1 with weight $-\weight_{\min}$; that is, unless $H^0(X,L)$ as a $\hU$-module has a direct summand on which $U$ acts trivially and $\CC^*$ acts with weight $-\weight_{\min}$. Equivalently each point in the projective space $\PP(H^0(X,L)^*)$ represented by a weight vector of minimal weight $\weight_{\min}$ has trivial stabiliser in $U$. 

 Thus when $U=\CC^+$ has dimension one  and $X=\PP(H^0(X,L)^*)$ is a projective space, the condition that semistability coincides with stability is equivalent to the requirement that every $x \in Z_{\min}$ has trivial stabiliser in $U$ (cf. \cite{BDHK16}).
\end{rem}

\begin{defn} \label{defn:welladapteddimone}
Let $\chi: \hat{U} \to \CC^*$ be a character of the semi-direct product $\hat{U}= \CC^* \ltimes U$ acting linearly on $X$ as above.  Suppose that $ \weight_{\min} = \weight_0 < \weight_{1} < 
\cdots < \weight_{h} = \weight_{\max} $ are the weights with which the one-parameter subgroup $\CC^* \leq \hU$ acts on the fibres of the tautological line bundle $\calo_{\PP((H^0(X,L)^*)}(-1)$ over points of the connected components of the fixed point set $\PP((H^0(X,L)^*)^{\CC^*}$ for the action of $\CC^*$ on $\PP((H^0(X,L)^*)$.
We assume that there exist at least two distinct such weights since otherwise the action of $U$ on $X$ is trivial.
Let $c$ be a positive integer such that 
$$  \weight_{\min} = \weight_{0} <
\frac{\chi}{c} < \weight_{1}; $$
 such rational characters $\chi/c$  are called {\it  adapted } to the linear action of $\hU$ on $L$, and the linearisation is said to be adapted if $ \weight_{0} <0 < \weight_{1} $.
 Recall that the linearisation of the action of $\hat{U}$ on $X$ with respect to the ample line bundle $L^{\otimes c}$ can be twisted by the character $\chi$ so that the weights $\weight_j$ are replaced with $\weight_jc-\chi$;
let $L_\chi^{\otimes c}$ denote this twisted linearisation. Note that the unipotent group $U$ is contained in the kernel of $\chi$ and so the restriction of the linearisation to the action of $U$ is unaffected by this twisting.
Let $X^{s,\CC^*}_{\min+}$ denote the stable subset of $X$ for the linear action of $\CC^*$ with respect to the linearisation $L_\chi^{\otimes c}$, so that 
$$X^{s,\CC^*}_{\min+} = X^0_{\min} \setminus Z_{\min}
$$
where 
$$Z_{\min} = \{ x \in X^{\CC^*} \, | \, \CC^* \mbox{ acts on $L^*|_x$ with weight }  \weight_{\min} \} $$
and
$$ X^0_{\min} = \{ x \in X \, | \, \lim_{t \to 0} \, tx \in Z_{\min} \}
.$$
Let $$X^{s,\hU}_{\min+} = X \setminus \hU (X \setminus X^{s,\CC^*}_{\min+}) = \bigcap_{u \in U} u X^{s,\CC^*}_{\min+}.$$ 
Finally we say that a property holds for a linear action of $\hat{U}$ with respect to a linearisation twisted by a {\it well adapted} rational character if there exists $\epsilon >0$ such that if $\chi/c$ is any rational character of $\CC^*$ (lifted to $\hU$) with 
$$  \weight_{\min} <
\frac{\chi}{c} < \weight_{\min} + \epsilon$$
then the property holds for the induced linearisation on $L^{\otimes c}$ twisted by $\chi$.
\end{defn}

The aim of this section is to prove the following theorem, which we will use for our inductive proof of Theorem \ref{mainthm}.

\begin{theorem} \label{fingendimone}
Let $X$ be a complex projective variety equipped with a linear action
(with respect to an ample line bundle $L$) of  a
semi-direct product $\hat{U} = \CC^* \ltimes \CC^+$, where
 the weight of
the induced $\CC^*$ action on the Lie algebra of
$U=\CC^+$ is strictly positive.
Suppose that  the linear action of $\hU$ on $X$ satisfies the condition that \lq semistability coincides with stability' as above. If  $\chi: \hU \to \CC^*$ is a character of $\hat{U}$ and $c$ is a sufficiently divisible positive integer such that the rational character $\chi/c$ is adapted for the linear action of $\hU$ with respect to $L$, then after twisting this linear action by $\chi/c$ we have
\begin{enumerate}
\item $X^{s,\hU}_{\min+}$ is a $\hU$-invariant open subvariety  of $X$ with a geometric quotient $\pi: X^{s,\hU}_{\min+} \to
X^{s,\hU}_{\min+}/\hU$ by the action of $\hU$;
\item this geometric quotient $X^{s,\hU}_{\min+}/\hU$ is a projective variety and the  tensor power $L^{\otimes c}$ of $L$ descends to a very ample line bundle 
 on $X^{s,\hU}_{\min+}/\hU$;
\item $X^{ss,\hU}  = X^{s,\hU}_{\min+}  $ and the geometric quotient $X^{s,\hU}_{\min+}/\hU$ is the enveloping quotient $X\env_{L_\chi^{\otimes c}} \hU$;
\item the algebra of invariants 
$\bigoplus_{k \geq 0} H^0( X,L_{k\chi}^{\otimes ck})^{\hU}$ is finitely generated and the enveloping quotient $X\env_{L_\chi^{\otimes c}} \hU  \cong \mathrm{Proj}(\oplus_{k \geq 0} H^0( X,L_{k\chi}^{\otimes ck})^{\hU})$ is the associated projective variety;
\item the  tensor power $L^{\otimes c}$ of $L$ induces a very ample line bundle on an inner enveloping quotient $X \inenv U$ for the action of $U$ on $X$ with a $\CC^*$-equivariant embedding
$$X \inenv U \to \PP((H^0(X,L^{\otimes c})^U)^*)$$
as a quasi-projective subvariety, containing the geometric quotient $X^{s,U}/U$ as an open subvariety, with closure $\overline{X \inenv U}$ in  $\PP((H^0(X,L^{\otimes c})^U)^*)$;
\item $X^{s,\hU}_{\min+}$ is a $U$-invariant open subvariety of $X^{s,U}$ and has a geometric quotient $X^{s,\hU}_{\min+}/U$  which coincides with both the stable and semistable loci $(\overline{X \inenv U})^{s,\CC^*} = (\overline{X \inenv U})^{ss,\CC^*}$ for the $\CC^*$ action with respect to the linearisation on $\calo_{\PP((H^0(X,L^{\otimes c})^U)^*)}(1)$ induced by ${L_\chi^{\otimes c}}$, so that the associated GIT quotient of $\overline{X \inenv U}$ by $\CC^*$ is given by
$$\overline{X \inenv U}/\!/ \CC^* \cong (X^{s,\hU}_{\min+}/U)/\CC^* \cong X^{s,\hU}_{\min+}/\hU = X\env_{L_\chi^{\otimes c}} \hU.$$
\end{enumerate} 
\end{theorem}

\begin{rem}
Note that when $\dim (U) = 1$ the hypothesis that the linearisation is adapted is sufficient; we do not need the stronger requirement that the linearisation is well adapted (cf. Remark \ref{wellvsadapted}).
\end{rem}

In order to prove Theorem \ref{fingendimone}, we will first  prove the theorem in the case where
$X=\PP(V)$ and $L=\OO_{\PP(V)}(1)$, for a finite dimensional
$\hat{U}$-representation $V$. This is done by an application of Theorem \ref{thm:geomcor} in a very simple situation with $G=SL(2)$ and $R=\CC^*$.
 An explicit description of the stable locus is used to
prove that
the enveloping quotient map $\PP(V)^{ss, \hU} \to
\PP(V) \env  \hat{U}$ is a geometric quotient for the
$\hat{U}$-action on $\PP(V)^{ss, \hU}$.  Theorem \ref{fingendimone} then follows by embedding $X$
into a projective space and analysing the behaviour of stability under closed immersions (cf. Remark \ref{rmk:GiSt2}).


\subsection{The case $(X,L)=(\PP(V),\OO(1))$}
\label{subsection:ExAdCase}

Let $V$ be a finite-dimensional
representation of $$\hat{U}=\hU^{[\ell]} = U \rtimes  \CC^*$$
where $U = \CC^+$ and $\CC^*$ acts on $\Lie(U)$ with weight $\ell \geq 1$, and let $X=\PP(V)$ with
$L=\OO(1)$ defining the canonical linearisation. 

\begin{defn} Let $V_{\min}$ be the $\CC^*$-weight space in $V$ of minimal weight $ \weight_{0}$, and let
$\PP(V)^0_{\min}$ be the open subvariety of $\PP(V)$ consisting of points flowing to
$Z_{\min} = \PP(V_{\min})$ under the action of $t \in \CC^*$ as $t \to 0$.
\label{def:ExPr2}
\end{defn} 

\begin{rem} Recall that in this situation the condition that semistability coincides with stability given in Definition \ref{defn:ss=sdimone} is equivalent to saying that $V_{\min}$ does not contain
any fixed points for the $\CC^+ $-action on $V$, and that  $X^{s,\hU}_{\min+} = \PP(V)^0_{\min} 
\setminus (U \cdot \PP(V_{\min}))$.
\end{rem}

We wish to prove the following proposition. 

\begin{prop} \label{prop:ExAd1}
 If $V_{\min}$ does not contain
any fixed points for the $\CC^+ $-action on $V$, and the linearisation is twisted by an adapted rational character $\chi/c$, then
\begin{enumerate}
\item \label{itm:ExAd1-1} there are equalities
  $\PP(V)^{\rms,\hU} =\PP(V)^{ss, \hU}=\PP(V)^0_{\min} 
\setminus (U \cdot \PP(V_{\min}))$;  
\item \label{itm:ExAd1-2} the enveloping quotients $\PP(V) \env \, \hat{U}$ and
  $\PP(V) \env  \,U$ are projective varieties, and for suitably divisible 
integers $r>0$ the algebras of invariants $\bigoplus_{k\geq 0} H^0(X,L^{\otimes
  kr})^{\hat{U}}$ and $\bigoplus_{k\geq 0} H^0(X,L^{\otimes kr})^U$ are finitely
generated; and  
\item \label{itm:ExAd1-3} the enveloping
quotient map $\PP(V)^{ss,\hU}  \to \PP(V) \env  \hat{U}$  
is a geometric quotient for the $\hat{U}$-action on
$\PP(V)^{ss,\hU} $.  
\end{enumerate}
\end{prop}

 In
order to study the linear action of $\hat{U}=\hat{U}^{[\ell]}$  we consider the surjective
homomorphism 
\[
\eta_{\ell}:\hat{U}^{[2\ell]} \to \hat{U}^{[\ell]}, \quad (u;t) \mapsto (u;t^2).
\] 
We can pull back the linear action of  $\hat{U}=\hat{U}^{[\ell]}$ to a linear action of $\hat{U}^{[2\ell]}$ via
$\eta_{\ell}$. The (semi)stable loci for the
linear actions of $\hat{U}^{[\ell]} $ and
$\hat{U}^{[2\ell]} $ then coincide, and the same
is true for the enveloping quotients. 
In order to 
prove Proposition \ref{prop:ExAd1} we may therefore work with $\hat{U}^{[2\ell]}
$. 

Now the
$\hat{U}^{[2\ell]}$-representation $V$ defined by $\eta_{\ell}$ admits a
decomposition  
\begin{equation} \label{eq:ExAd1.3}
V\overset{\hat{U}^{[2\ell]}}{\cong} \bigoplus_{i=1}^q \kk^{(a_i)} \ten \sym^{l_i}
\kk^2,   
\end{equation}
of $\hat{U}^{[2\ell]}$-modules, where 
\begin{itemize}
\item $\kk^{(a_i)}$ is the one dimensional representation of $\hat{U}^{[2\ell]}$
  defined by the character $\hat{U}^{[2\ell]} \to \CC^*$ of weight
$a_i \in \mb{Z}$; 
\item $\sym^{l_i} \kk^2$ is the standard irreducible
representation of $G=\SL(2,\kk)$ of highest weight $l_i \geq 0$, on which $\hat{U}^{[2\ell]}$
acts via the surjective homomorphism
\[
\rho_{\ell}:\hat{U}^{[2\ell]} \to \hat{U}^{[2]}, \quad (u;t) \mapsto (u;t^{\ell})
\]
and the identification of $\hat{U}^{[2]}$ with the Borel subgroup $B \subseteq
G$ of upper triangular matrices given by
\[
\hat{U}^{[2]}= \CC^+  \rtimes \CC^*  \to B, \quad (u;t) \mapsto
\left( \begin{smallmatrix} t & tu \\ 0 & t^{-1} \end{smallmatrix} 
\right);    
\]
and
\item because the action of $\hat{U}^{[2\ell]}$ factors through
  $\eta_{\ell}:\hat{U}^{[2\ell]} \to \hat{U}^{[\ell]}$, we have $a_i \equiv \ell
  \, l_i \mod 2$ for each $i=1,\dots,q$.    
\end{itemize}

  Observe that
$\rho_{\ell}:\hat{U}^{[2\ell]} \to \hat{U}^{[2]} \cong B \subseteq G=\SL(2,\kk)$ restricts
to give the standard 
inclusion of the unipotent radical $U=\CC^+  = (\hat{U}^{[2\ell]})_u$ of $\hat{U}^{[2\ell]}$ inside $G$ as the subgroup of
strictly upper triangular matrices, so $\rho_{\ell}$ is an
$(\hat{U}^{[2\ell]})_u$-faithful homomorphism, in the sense of Definition 
\ref{def:GiSt1.3}. The linear action of $U$ on $V$ extends to a linear action of $G$
by demanding that $G$ act on
$\sym^{l_i}\kk^2$ in the usual manner and 
trivially on $\kk^{(a_i)}$ for each $i$.    

There is therefore
an isomorphism of  $G \times \CC^*$-spaces (where
$\CC^* = \hat{U}^{[2\ell]}/U$) 
\[
G \times_{U} \PP(V) \cong (G/U) \times \PP(V)
\]
which lifts to an isomorphism of linearisations. Here the action of $G \times \CC^*$ on the quasi-affine variety $G/U$ is left multiplication by $G$ and right multiplication by $\CC^*$, and   the  $G \times
\CC^*$-linearisation $\mc{P}$ on $\PP(V)$ with respect to  $\OO(1)$
is given by the $G$-action above and the following linear action of $\CC^*$, twisted twice by $ \chi/c$  :
\begin{align} \label{eq:ExAd2.1}
t\cdot v=\textstyle{\sum_i} (t^{a_i}z_i)
\ten s_i, \quad &v \in V, \ t \in \CC^*, \\
&v=\textstyle{\sum_i} z_i \ten s_i \in \bigoplus_{i=1}^q
\kk^{(a_i)} \ten \sym^{l_i} \kk^2 \nonumber \\
&\text{\upshape{via the isomorphism} \eqref{eq:ExAd1.3}.} \nonumber
\end{align}

The homomorphism $\rho_{\ell}$ embeds $U$ into $G=\SL(2;\CC)$ as a
Grosshans subgroup, since there is
an isomorphism $G/U \cong \kk^2 
\setminus \{0\}$ given by considering the orbit of $\left(\begin{smallmatrix} 1 \\
    0 \end{smallmatrix}\right) \in \kk^2$ under the defining representation of
$G$, and the inclusion $\kk^2 \setminus \{0\} \hookrightarrow \kk^2$ defines a
nonsingular affine 
completion of $G/U$ with codimension 2
complement. We may therefore  construct a  $G \times \CC^*$-equivariant
nonsingular projective completion $\PP^2=\{[v_0:v_1:v_2] 
\}$ of $G/U$ by  adding a hyperplane at infinity defined
by $v_0=0$ to 
$\kk^2$. The
action of $G \times \CC^*$ on $\PP^2=\PP(\kk^3)$
is defined by the representation given in block form by
\[
(g,t) \mapsto \left(\begin{array}{c|c} 1 & 0 \\ \hline 0 &
    g\left(\begin{smallmatrix}t^{-\ell} & 0 \\ 0 &
        t^{-\ell} \end{smallmatrix}\right) \end{array}\right)  
\in \GL(3;\kk), \quad g \in G, \ t \in \CC^*,   
\]
where $\GL(3;\kk)$ acts on $\kk^3$ by left multiplication. For any integer
$N>0$, this representation 
 determines a $G \times \CC^*$-linearisation on
$\OO_{\PP^2}(N) \to \PP^2$ which restricts
to the canonical linearisation on $\OO_{G/U} \to G/U$. Let $\mc{P}^{\prime}_{N} $ denote the  $G \times \CC^*$-linearisation 
over $ \PP^2 \times \PP(V)$ given by tensoring this with $\mc{P}$.
As $U$ is a Grosshans subgroup of $G$,  the
algebras of $U$-invariants and $\hat{U}$-invariants of any positive tensor power of this  
linearisation  are finitely generated
$\kk$-algebras, and the corresponding enveloping quotients    
\[
\PP(V) \env U \cong (\PP^2 \times \PP(V))
\dblslash_{\mc{P}^{\prime}_{N}} G, \quad \PP(V)
\env \hat{U}^{[2\ell]} \cong (\PP^2 \times \PP(V)) 
\dblslash_{\mc{P}^{\prime}_{N}} (G \times \CC^*) 
\]
are projective varieties.
By Theorem \ref{thm:geomcor} the stable loci $\PP(V)^{\rms,U}$ and
$\PP(V)^{\rms,\hU}$ and finitely generated 
semistable loci $\PP(V)^{ss,U}$ and $\PP(V)^{ss,\hU}$  may be computed as the
 completely (semi)stable loci  for 
the $G$ and  $G \times\CC^*$-linearisation
$\mc{P}^{\prime}_{N}$ (using
the Hilbert-Mumford criteria) and   Proposition
\ref{prop:ExAd1} is a consequence of the following lemma. 

\begin{lemma} \label{lem:ExAd1.2}
Under the hypotheses of Proposition \ref{prop:ExAd1},
stability and semistability are equivalent for the 
linear action of $G \times \CC^*$ on $\PP^2 \times \PP(V)$ with respect to the linearisation $\mc{P}^{\prime}_N$
when $N>\!>1$. Moreover if a point $p \in \PP^2 \times \PP(V)$ is stable for this linear action of $G \times \CC^*$ then $p \in (\kk^2 \setminus \{0\})
\times \PP(V)$, and if $p=([1:1:0],[v])$  then $p$ is stable if and only if  $[v] \in \PP(V)^0_{\min} \setminus U\PP(V_{\min})$.
\end{lemma}

\begin{proof} 
We shall deduce this by using the Hilbert-Mumford criteria as given in Proposition
\ref{sss} using the maximal torus $T_1 \times T_2 \subseteq G
\times 
\CC^*$, where $T_1$ is the subgroup of diagonal matrices
in $G$ and $T_2=\CC^*$. The group of characters of $T_1
\times T_2$ is identified with $\mb{Z} \times \mb{Z}$ in the natural
way. Introduce the following notation: for $i=1,\dots,q$ 
let $e_{i,1} 
,e_{i,2} 
$ be the standard basis of $\kk^2$, so that
\[
e_{i,1}^{l_i},\dots,e_{i,1}^{j}e_{i,2}^{l_i-j},\dots,e_{i,2}^{l_i} \in \sym^{l_i} \kk^2
\]
form a basis of $T_1 \times T_2$-weight vectors in $\sym^{l_i}\kk^2$. 
The fixed points in $\PP^2
\times \PP(V)$ for the $T_1 \times T_2$-action, along with the 
corresponding rational weights for $\mc{P}^{\prime}_N$,
are given in Table \ref{tab:ExAd1}.
\begin{table}[h]
\begin{center}
\begin{tabular}{c|c} Fixed point & Rational weight in \\ 
($j=0,\dots,l_i$, $i=1,\dots,q$) & $\Hom(T_1 \times T_2,\CC^*) \ten_{\mb{Z}}
\mb{Q} =\mb{Q} \times \mb{Q}$ \\ \hline 
$([1:0:0],[1 \ten e_{i,1}^{j}e_{i,2}^{l_i-j}])$ & $(2j-l_i,a_i-2 \weight_{0}-2\epsilon)$ \\
$([0:1:0],[1 \ten e_{i,1}^{j}e_{i,2}^{l_i-j}])$ &
$(2j-l_i,a_i-2 \weight_{0}-2\epsilon)+(N,-\ell N)$ \\ 
$([0:0:1],[1 \ten e_{i,1}^{j}e_{i,2}^{l_i-j}])$ & $(2j-l_i,a_i-2 \weight_{0}-2\epsilon)
+(-N,-\ell N)$
\end{tabular}
\end{center}
\caption{Rational weights of the fixed points of $T_1 \times T_2 \act \PP^2
  \times \PP(V)$ with respect to the linearisation $\mc{P}^{\prime}_{N}$.}
\label{tab:ExAd1}
\end{table}

The minimal $\CC^*$-weight for the $\hU$-action on $V$ is
$$\weight_{\min} = \min\{(a_i - \ell l_i)/2 \mid i = 1, \ldots ,q\}.$$
Let us temporarily call an index $i \in \{0,\ldots ,q\}$ {\it exceptional} if $\weight_{\min} = (a_i - \ell l_i)/2$.

Consider the rational weight $\vartheta=(2j-l_i,a_i-2 \weight_{0}-2\epsilon)$ for the
fixed point $([1:0:0],[1 \ten e_{i,1}^{j}e_{i,2}^{l_i-j}])$. Note that either
$\vartheta$ is contained in the interior of the cone  
\[
C=\{(c_1,c_2) \in \mb{Q}_{\geq 0} \times \mb{Q}_{\geq 0} \mid \ell \, c_1+c_2 \geq 0 
\text{ and } -\ell \, c_1+c_2 \geq 0 \},
\]
or $\vartheta$ lies outside $C$ and $i,j$ satisfy $ \weight_{0}=(a_i-\ell \, l_i)/2$
and $j \in \{0,l_i\}$: as
$0<\epsilon < 1/2$ we see that 
\[
\ell(2j-l_i)+(a_i-2 \weight_{0}-2\epsilon) \begin{cases}
=-2\epsilon < 0 & \text{iff } j=0 \text{ and }  \weight_{0}=(a_i-\ell \, l_i)/2 \\
>0 & \text{otherwise}
\end{cases}
\]
while 
\[
-\ell(2j-l_i)+(a_i-2 \weight_{0}-2\epsilon) \begin{cases}
=-2\epsilon < 0 & \text{iff } j=l_i \text{ and }  \weight_{0}=(a_i-\ell \, l_i)/2 \\
>0 & \text{otherwise}.
\end{cases}.
\]
We also claim that $a_i-2 \weight_{0}-2\epsilon>0$ for all $i=1,\dots,q$. Indeed,
suppose for a contradiction that $a_i-2 \weight_{0}-2\epsilon \leq 0$ for some $i=1,\dots,q$. Because
$0<2\epsilon<1$ and $a_i-2 \weight_{0} \in \mb{Z}$, this is equivalent to $a_i-2 \weight_{0} \leq
0$. But $2  \weight_{0} \leq a_i-\ell\, l_i$, so $\ell \, l_i \leq a_i-2 \weight_{0} \leq
0$. Because $\ell > 0$ we must have $l_i=0$, and by examining the 
possible cases for the value of $\ell(2j-l_i) +(a_i-2 \weight_{0} -2\epsilon)$ we see
that $ \weight_{0}=a_i/2$ and $i$ is exceptional. This implies there is a line
$\kk^{( \weight_{0})} = \kk^{( \weight_{0})} \ten \sym^0 \kk^2 \subseteq V_{\min}$ fixed by
$U$, which contradicts the assumption that
$V_{\min}$ does not 
contain a point fixed by the $U$-action, and so the claim is verified.  

\begin{figure}[ht]
\begin{tikzpicture}[scale=1] \draw[->] (0,-8) -- (0,4) node[anchor=south]{$\mb{Q} \cong \Hom(T_2,\CC^*)\ten_{\mb{Z}} \mb{Q}$}; \draw[->] (-7,0) -- (7,0) node[anchor=south east] {$\mb{Q} \cong \Hom(T_1,\CC^*)\ten_{\mb{Z}} \mb{Q}$}; \draw[->] (0,0) -- (1.5,3) node[anchor=west]{$\left(\begin{smallmatrix} 1 \\ \ell \end{smallmatrix} \right)$}; \draw[->] (0,0) -- (-1.5,3) node[anchor=east]{$\left(\begin{smallmatrix} -1 \\ \ell \end{smallmatrix} \right)$}; \foreach \x in {-0.45,-0.15,0.15,0.45}{  -0.15+0.3n for n=-1,0,1,2 \fill (\x,0.55) circle (1.5pt); 
};
 \foreach \x in {-0.15,0.15}{ \fill (\x,1.05) circle (1.5pt); }; \foreach \x in {-0.9,-0.6,...,0.9}{  0.3n for n=-3,-2,...,3  \fill (\x,1.55) circle (1.5pt); };  \fill (0,2.25) circle (1.5pt); 
 \foreach \x in {-1.05,-0.75,-0.45,-0.15,0.15,0.45,0.75,1.05}{ -0.15+0.3n for n=-2,-1,...,3  \fill (\x,2.55) circle (1.5pt); 
 };  
 \begin{scope}[shift={(-3,-6)}] \draw[->] (0,0) -- (1.5,3); \draw[->] (0,0) -- (-1.5,3); \node at (0,0) [anchor=north]{$(-N,-\ell N)$}; \foreach \x in {-0.45,-0.15,0.15,0.45}{  -0.15+0.3n for n=-1,0,1,2 \fill (\x,0.55) circle (1.5pt); };  \foreach \x in {-0.15,0.15}{  \fill (\x,1.05) circle (1.5pt); 
 }; 
\foreach \x in {-0.9,-0.6,...,0.9}{  0.3n for n=-3,-2,...,3 \fill (\x,1.55) circle (1.5pt); };  \fill (0,2.25) circle (1.5pt); 
\foreach \x in {-1.05,-0.75,-0.45,-0.15,0.15,0.45,0.75,1.05}{ -0.15+0.3n for n=-2,-1,...,3  \fill (\x,2.55) circle (1.5pt); 
}; \end{scope} 

bottom right chamber (translate top chamber by (3,-6))

\begin{scope}[shift={(3,-6)}]
\draw[->] (0,0) -- (1.5,3); \draw[->] (0,0) -- (-1.5,3); \node at (0,0) [anchor=north]{$(N,-\ell N)$}; 
\foreach \x in {-0.45,-0.15,0.15,0.45}{  -0.15+0.3n for n=-1,0,1,2 \fill (\x,0.55) circle (1.5pt); 
};  
\foreach \x in {-0.15,0.15}{  \fill (\x,1.05) circle (1.5pt); 
}; 
\foreach \x in {-0.9,-0.6,...,0.9}{  0.3n for n=-3,-2,...,3  \fill (\x,1.55) circle (1.5pt); 
};  \fill (0,2.25) circle (1.5pt); 

\foreach \x in {-1.05,-0.75,-0.45,-0.15,0.15,0.45,0.75,1.05}{ -0.15+0.3n for n=-2,-1,...,3  \fill (\x,2.55) circle (1.5pt); 
}; \end{scope} \end{tikzpicture}
\caption{Example of distribution of rational weights for $T_1 \times T_2 \act \mc{P}^{\prime}_{N} \to
  \PP^2 \times 
\PP(V)$.}
\label{fig:ExAd1}
\end{figure}
Thus  for sufficiently large $N>0$ the weights for the rational
$T_1 \times T_2$-linearisation $\mc{P}^{\prime}_N \to \PP^2 \times \PP(V)$ are
arranged in the fashion of 
Figure \ref{fig:ExAd1}.  In particular, the
weight polytope $\Delta_p 
\subseteq \Hom(T_1 \times T_2,\CC^*) \ten_{\mb{Z}} \mb{Q}$
for a point $p=([w_0:w_1:w_2],[v]) \in \PP^2 \times \PP(V)$ contains the origin
precisely when the interior $\Delta_p^{\circ}$ does and so semistability and
stability for $T_1 \times T_2$ coincide. 

By the Hilbert-Mumford criteria, the point $p$ is (semi)stable for the  $G \times\CC^*$-linearisation if and only if 
$gp$ is $T_1 \times T_2$-(semi)stable for each $g \in G$. It follows that stability and
semistability are equivalent for $G \times \CC^*$. 

Using the isomorphism \eqref{eq:ExAd1.3}, write
\[
v=\textstyle{\sum_{i=1}^q} z_i \ten s_i, \quad z_i \in \kk^{(a_i)}, \quad 0 \neq
s_i= \textstyle{\sum_{j=0}^{l_i}} v_{i,j}e_{i,1}^je_{2,i}^{l_i-j}, \ v_{i,j}
\in \kk.
\]
Then one finds that $p$ is $T_1 \times
T_2$-\emph{unstable} precisely when $p \notin (\kk^2 \setminus \{0\}) \times \PP(V)$
(i.e. $w_0=0$ or $w_1=w_2=0$) or else by satisfying one of the following criteria,
split into three cases: 
\begin{description}
\item[Case $w_0w_1w_2 \neq 0$] 
\[ 
    \begin{array}{rl} 
 \text{Either $v_{i,j} \neq 0       \implies$ $(i$ exceptional and $j=0)$,   }  
& \text{or $v_{i,j} \neq 0
      \implies$ $(i$ exceptional and $j=l_i)$.} 
    \end{array}
\]
\item[Case $w_0w_1 \neq 0$, $w_2=0$]
\[ 
    \begin{array}{rl} 
 \text{Either $i$ exceptional $\implies       v_{i,0}=0$,  }  
 & \text{  or $v_{i,j} \neq 0
      \implies$ $(i$ exceptional and $j=0)$.} 
    \end{array}
\]
\item[Case $w_0w_2 \neq 0$, $w_1=0$]
\[ 
    \begin{array}{rl} 
 \text{Either $i$ exceptional $\implies       v_{i,l_i}=0$,}  
& \text{or $v_{i,j} \neq 0
      \implies$ $(i$ exceptional and $j=l_i)$.} 
    \end{array}
\]
\end{description}
It follows that  if $p$ is $G\times \CC^*$-stable then $p \in (\kk^2 \setminus \{0\})
\times \PP(V)$, and it remains to show that if $p=([1:1:0],[v])$  then $p$ is $G\times \CC^*$-stable if and only if  $[v] \in \PP(V)^0_{\min} \setminus U\PP(V_{\min})$, or equivalently that $gp = ([1:g_{11}:g_{21}],[gv])$ (where $g = ( g_{ij})$) is $T_1 \times T_2$-stable for all $g \in G$ if and only if  $[v] \in \PP(V)^0_{\min} \setminus U\PP(V_{\min})$. 

Under the
isomorphism of vector spaces $V \cong \bigoplus_{i=1}^q \kk^{(a_i)} \ten \sym^{l_i}
\kk^2 $ the
weight vectors for the induced $\CC^* \subseteq \hat{U}^{[\ell]}$-action on $V$
take the 
form $1 \ten e_{1,i}^j e_{2,i}^{l_i-j}$, where $1 \leq i \leq q$ and $0 \leq  j
\leq l_i$, with the weight of $1 \ten e_{1,i}^j e_{2,i}^{l_i-j}$ equal to
$(a_i-\ell \, l_i +2j)/2 \in \mb{Z}$. The weight space $V_{\min}$ of
minimal weight $ \weight_{0}$ is spanned by all $1 \ten e_{2,i}^{l_i}$ with $i$ an
exceptional index, and the $U$-sweep $U \cdot V_{\min}$ of $V_{\min}$ is
contained in the $\hat{U}^{[\ell]}$-subspace
\[
\bigoplus_{\text{$i$ exceptional}} \kk^{(a_i)} \ten \sym^{l_i}\kk^2 \subseteq V.
\]
Now, if $v=\sum_i z_i \ten s_i$ with each $s_i \neq 0$, then the existence of an
exceptional $i$ with $z_i \neq 0$ and $s_i$ not divisible by
$(1,0)$ is equivalent to $\lim_{t \to 0} t\cdot [v] \in \PP(V_{\min})$ (where we
take $t
\in \CC^* \subseteq \hat{U}^{[\ell]}$ in the limit). 

 The existence of a non-exceptional $i$ such that $z_i \neq 0$
is equivalent to $v \notin \bigoplus_{\text{$i$ exceptional}} \kk^{(a_i)} \ten
  \sym^{l_i}\kk^2$, which itself implies $[v] \notin U \cdot
  \PP(V_{\min})$. On the other hand, because of the transitivity of the
  $U$-action on $\kk=\PP^1 \setminus \{[1:0]\}$ and the fact that $V_{\min}$ is
  spanned by $1 \ten e_{i,2}^{l_i}$ with $i$ exceptional, we see that $[v] \in U
  \cdot \PP(V_{\min})$ if and only if  $v \in \bigoplus_{\text{$i$
      exceptional}} \kk^{(a_i)} \ten 
  \sym^{l_i}\kk^2 $ and there is some $(w_1,w_2) \in \kk^2 \setminus
  \{0\}$ with $[w_1:w_2] \neq [1:0] \in \PP^1$ such that $s_i =
  (w_1,w_2)^{l_i} \in \sym^{l_i} \kk^2$ for all exceptional $i$.

Therefore, applying the three cases above to $gp = ([1:g_{11}:g_{21}],[gv])$ where $p=([1:1:0],[v])$, we deduce that $p$ is $G\times \CC^*$-stable if and only if  $[v] \in \PP(V)^0_{\min} \setminus U\PP(V_{\min})$,
as required. \end{proof}


This completes the proof of Proposition \ref{prop:ExAd1}, and of Theorem \ref{fingendimone} in the special case when $(X,L) = (\PP(V), \OO(1))$.

\subsection{Proof of Theorem \ref{fingendimone} for general $(X,L)$}
\label{subsec:ExAdProof}

 Suppose now that  $L$ is a very ample line bundle over an irreducible projective
variety $X$ equipped with a $\hat{U}$-linearisation, where $\hat{U} = U \rtimes \CC^*$for $U=\CC^+ $, let $V=H^0(X,L)^*$ and
let $\gamma:X \hookrightarrow \PP(V)$ be the canonical closed immersion. Let
$ \weight_{0}$ be the minimal weight for the induced $\CC^*$-action on $V$ and suppose
the associated weight space $V_{\min}$ does not contain any fixed points for the
$U$-action on $V$. Finally, let $\chi/c$
be an adapted rational character. 

By Proposition \ref{prop:ExAd1} there is an
enveloping quotient 
\[
q:\PP(V)^{ss,\hU}=\PP(V)^{s,\hU}
\to \PP(V) \env \hat{U}
\]
which is a geometric quotient for the $\hat{U}$-action on
$\PP(V)^{ss,\hU}$, and the quotient $\PP(V) \env \hat{U}$ is a projective
variety. Furthermore, 
\[
\PP(V)^{s,\hU}=\PP(V)^0_{\min} \setminus (U \cdot
\PP(V_{\min})),
\]
from which it follows that
\[
\PP(V)^{s,\hU} \cap X = \gamma^{-1}(\PP(V)^{s,\hU}) = X^0_{\min} \setminus (U
\cdot Z_{\min}).
\]
Since geometric quotients behave well under closed immersions, it follows from the definitions of (semi)stability that $\PP(V)^{s,\hU} \cap X \subseteq X^{ss,\hU}$, and thus that 
$X^0_{\min} \setminus (U \cdot Z_{\min})$ is an open subvariety of $
X^{ss,\hU}$ whose image under the enveloping quotient
\[
q:X^{ss,\hU} \to X \env \hat{U}
\]
is a geometric quotient for the $\hat{U}$-action on
$X^0_{\min} \setminus (U \cdot Z_{\min})$ that 
embeds naturally as a closed subvariety of
$\PP(V)^{s,\hU}/\hat{U}=\PP(V) \env \hat{U}$, which is projective. Hence $q(X^0_{\min} \setminus (U
\cdot Z_{\min}))$ is itself a projective variety. In particular it is complete,
and since $X \env \hat{U}$ is separated over $\spec \kk$ it
follows that the inclusion $q(X^0_{\min} \setminus (U \cdot Z_{\min}))
\hookrightarrow X \env \hat{U}$ is a 
closed map \cite[Tag 01W0]{stacks-project}. 
On the other hand,
because $X$ is irreducible 
$q(X^0_{\min} \setminus (U \cdot Z_{\min}))$ is a dense open subset of $X
\env \hat{U}$ if $X^0_{\min} \setminus (U \cdot Z_{\min})$ is not empty, while it is easy to see that $X\env \hU = \emptyset$ if 
$X^0_{\min} \setminus (U \cdot Z_{\min}) = \emptyset$.
 Hence 
\[
q(X^0_{\min} \setminus (U \cdot Z_{\min}))=X \env \hat{U},
\]
 $X \env \hat{U}$ is a projective variety, 
and Theorem \ref{fingendimone} now follows from 
Theorem \ref{thm:geomcor} and 
Proposition \ref{cor:GiFi5}, together with Proposition
\ref{prop:ExAd1} which covers the case when $(X,L)=(\PP(V),\OO_{\PP(V)}(1))$.

\section{Actions of $\CC^*$-extensions of unipotent groups} 


Now let $U$ be any graded unipotent group; that is, $U$ is a
unipotent group with a one-parameter group of automorphisms
$\lambda:\CC^* \to \mbox{Aut}(U)$ such that the weights of
the induced $\CC^*$ action on the Lie algebra $\lieu$ of
$U$ are all strictly positive.  Let $\hU$ be the corresponding semidirect
product
$$\hat{U} = \CC^* \ltimes U.$$
As before we assume that $\hU$ acts linearly on a projective variety $X$ with respect to a very ample line bundle $L$, and 
 that  the action of $U$ on $X$ is not trivial.
Recall from Definition \ref{defn:ss=sdimone} that we say that {\it semistability coincides with stability} for this linear action if whenever $U'$ is a subgroup of $U$ normalised by $\CC^*$ and $\xi$ belongs to the Lie algebra of $U$ but not the Lie algebra of $U'$ and $\xi$ is a weight vector for the action of $\CC^*$, then the weight space with weight $- \weight_{min}$ for the action of $\CC^*$ on $H^0(X,L)$ is contained in the image $\delta_\xi(H^0(X,L)^{U'})$ of $H^0(X,L)^{U'}$ under the infinitesimal action $\delta_\xi: H^0(X,L) \to H^0(X,L)$ of $\xi$ on $H^0(X,L)$.

We are aiming to prove Theorem \ref{mainthm} by induction on $\dim(U)$; the base case when $\dim(U) = 1$ has been proved in Theorem \ref{fingendimone}. The following lemma will be used for the induction step.

\begin{lemma} \label{lemtensorpower}
Suppose that a linear action of $\hat{U}$ on a projective
variety $X$ with respect to an ample line bundle $L$ satisfies the condition that semistability equals stability. If $U_{\dagger}$ is a subgroup of $U$ which is normal in $\hU$ and $\hU_{\dagger}  = \CC^* \ltimes U_{\dagger}$ is the subgroup of $\hU$ generated by $U_{\dagger}$ and the one-parameter subgroup $\CC^*$,
then
\begin{enumerate}
\item the linear action of $\hU$ on $X$ with respect to any positive tensor power $L^{\otimes m}$ of $L$  satisfies the condition that semistability equals stability;
\item the restriction to $\hU_{\dagger}$ of the linear action of $\hU$ on $X$ with respect to any positive tensor power $L^{\otimes m}$ of $L$ satisfies the condition that semistability equals stability;
\item if $c_{\dagger}$ is a sufficiently divisible positive integer  then the induced linear action of $\hU/U_{\dagger}$ on 
  the
closure $\overline{X \inenv U_{\dagger}}$ in  $\PP((H^0(X,L^{\otimes c_{\dagger}})^{U_{\dagger}})^*)$ of an inner enveloping quotient $X \inenv U_{\dagger}$ for the action of $U_{\dagger}$ on $X$
satisfies the condition that semistability equals stability with respect to the ample line bundle determined by $L^{\otimes c}$.
\end{enumerate}
\end{lemma}

\noindent {\bf Proof:}  (1) Suppose that $U'$ is a subgroup of $U$ normalised by $\CC^*$ and that a $\CC^*$-weight vector $\xi$ with weight $a$ belongs to the Lie algebra of $U$ but not the Lie algebra of $U'$, with corresponding infinitesimal action $\delta_\xi = \delta: H^0(X, L) \to H^0(X, L)$. 
By abuse of notation let $\delta$ also denote the induced infinitesimal action on $H^0(X,L^{\otimes m})$. 
As $X$ is $\CC^*$-invariant, the minimum weight $ \weight_{0}^{L^{\otimes m}} $ with which the one-parameter subgroup $\CC^* \leq \hU$ acts on the fibres of the line bundle $\calo_{\PP((H^0(X,L^{\otimes m})^*)}(-1)$ over points of the connected components of the fixed point set $\PP((H^0(X,L^{\otimes m})^*)^{\CC^*}$ for the action of $\CC^*$ on $\PP((H^0(X,L^{\otimes m})^*)$ is $m \weight_{0}$.
Suppose that $s \in H^0(X,L^{\otimes c})^{U'}$ is a weight vector with weight $ \weight_{0}^{L^{\otimes m}}$
 for the action of $\CC^*$. 
 We want  to show that  there is some section $s'  \in H^0(X,L^{\otimes m})^{U'}$ such that 
 $\delta(s') = s$.
 Since $ \weight_{0}^{L^{\otimes m}} = m \weight_{0}$
we can write $s$ as a linear combination of monomials $s_1\cdots s_m$ where
$s_j \in H^0(X,L)$ is a weight vector with weight $ \weight_{0}$ for the
$\CC^*$ action, which implies that $\delta(s_j) = 0$ for $j=1, \ldots , m$. 
As the linear action of $\hat{U}$ on  $X$ with respect to $L$ satisfies the
condition that semistability equals stability, 
 there is $s_1' \in H^0(X,L)^{U'}$ such that $\delta (s_1') =
 s_1$. It follows that   
$$ \delta(s_1's_2 \cdots s_m) = s_1\cdots s_m$$
where $s_1's_2 \cdots s_m \in H^0(X,L^{\otimes m})^{U'}$ as required.

(2) By (1) we can assume that $m = 1$ and then this follows straight from the definition of what it means for semistability to coincide with stability (Definition \ref{defn:ss=sdimone}).

(3) Recall from Remark \ref{remclaim} the notion of an inner enveloping quotient. A subgroup of $U/U_{\dagger}$ normalised by $\CC^*$ has the form $U'/U_{\dagger}$ where $U'$ is a subgroup of $U$ containing $U_{\dagger}$ and normalised by $\CC^*$. A weight vector in the Lie algebra of $U/U_{\dagger}$ which does not lie in the Lie algebra of $U'/U_{\dagger}$ can be represented by a weight vector $\xi$ in the Lie algebra of $U$ not lying in the Lie algebra of $U'$, and the corresponding infinitesimal action on $H^0(X,L^{\otimes c_{\dagger}})^{U_{\dagger}}$ is the restriction of the infinitesimal action on $H^0(X,L^{\otimes c_{\dagger}})$ determined by $\xi$, so (3) follows from (1).
\hfill $\Box$

\medskip

Our aim is to prove the following theorem, from which Theorem \ref{mainthm} and 
Corollary \ref{cor1.2} will follow.   

\begin{theorem} \label{fingen}
Let $X$ be a complex projective variety equipped with a linear action
(with respect to an ample line bundle $L$) of  a
unipotent group $U$ with a one-parameter group of automorphisms
 such that the weights of
the induced $\CC^*$ action on the Lie algebra of
$U$ are all strictly positive.
Suppose that the linear action of $U$
on $X$ extends to a linear action of the
semi-direct product $\hat{U} = \CC^* \ltimes U$. 
Suppose also that the linear action of $\hU$ on $X$ satisfies the condition that \lq semistability coincides with stability' as above. If  $\chi: \hU \to \CC^*$ is a character of $\hat{U}$ and $c$ is a sufficiently divisible positive integer such that the rational character $\chi/c$ is well adapted for the linear action of $\hU$ with respect to $L$, then after twisting this linear action by $\chi/c$ we have
\begin{enumerate}
\item the $\hU$-invariant open subset $X^{s,\hU}_{\min+}$ of $X$ has a geometric quotient $\pi: X^{s,\hU}_{\min+} \to
X^{s,\hU}_{\min+}/\hU$ by the action of $\hU$;
\item this geometric quotient $X^{s,\hU}_{\min+}/\hU$ is a projective variety and the  tensor power $L^{\otimes c}$ of $L$ descends to an ample line bundle $L_{(c,\hU)}$ on $X^{s,\hU}_{\min+}/\hU$;
\item $X^{s,\hU}_{L_\chi^{\otimes c}} = X^{ss,\hU}_{L_\chi^{\otimes c}} = X^{s,\hU}_{\min+} $; 
\item the geometric quotient $X^{s,\hU}_{\min+}/\hU$ is the enveloping quotient $X\env_{L_\chi^{\otimes c}} \hU$;
\item the algebra of invariants 
$\bigoplus_{k \geq 0} H^0( X,L_{k\chi}^{\otimes ck})^{\hU}$ is finitely generated and the enveloping quotient $X\env_{L_\chi^{\otimes c}} \hU  \cong \mathrm{Proj}(\oplus_{k \geq 0} H^0( X,L_{k\chi}^{\otimes ck})^{\hU})$ is the associated projective variety.
\end{enumerate} 
\end{theorem}

\begin{rem} \label{mindep}
Note that, for any positive integer $m$, Theorem \ref{fingen} holds for a linear action of $\hU$ on $X$ with respect to an ample line bundle $L$ if it holds for the induced linearisation of the action of $\hU$ with respect to the line bundle $L^{\otimes m}$. To see this, we use Lemma \ref{lemtensorpower} and observe that almost all the ingredients of the theorem are unchanged when $L$ is replaced with $L^{\otimes m}$. The only ingredients over which we still need to take care are the concept of well adaptedness and the definition of $X^{s,\hU}_{\min+}$, which both depend on the weights of the action of $\CC^*$ on $H^0(X,L)$ (see Definition \ref{defn:welladapteddimone}). However since $X$ is $\CC^*$-invariant the minimum weight $ \weight_{\min}^{L^{\otimes m}} = \weight_{0}^{L^{\otimes m}} $ with which the one-parameter subgroup $\CC^* \leq \hU$ acts on the fibres of the line bundle $\calo_{\PP((H^0(X,L^{\otimes m})^*)}(-1)$ over points of the connected components of the fixed point set $\PP((H^0(X,L^{\otimes m})^*)^{\CC^*}$ for the action of $\CC^*$ on $\PP((H^0(X,L^{\otimes m})^*)$ is $m \weight_{\min}$, and by variation of GIT for the reductive group $\CC^*$ we know that if $\weight_{\min} = \weight_{0} < \chi/c < \weight_{1}$ then the stable set $X^{s,\CC^*}_{\min+}$ for the linear action of $\CC^*$ with respect to the linearisation $L^{\otimes c}_\chi$ is the same as the stable set for the linear action of $\CC^*$ with respect to the linearisation $(L^{\otimes c}_\chi)^{\otimes m} = L^{\otimes cm}_{m\chi}$. So $ X^{s,\CC^*}_{\min+}$ and $X^{s,\hU}_{\min+} = X \setminus \hU (X \setminus X^{s,\CC^*}_{\min+}) $ are unchanged by replacing $L$ with $L^{\otimes m}$. Finally
$$ \chi/c =  \weight_{0} + \epsilon \,\,\,\,\, \mbox{ iff } \,\,\,\,\,\, m\chi/c = m \weight_{0} + m\epsilon =  \weight_{0}^{L^{\otimes m}} + m\epsilon$$
where $(L^{\otimes ck}_{k\chi})^{\otimes m} = L^{\otimes ckm}_{mk\chi} = (L^{\otimes m})^{\otimes ck}_{km\chi}$ for any $k \geq 0$. Thus $\chi/c$ is well adapted for the linear action of $\hU$ with respect to $L$ if and only if $m\chi/c$ is well adapted for the linear action of $\hU$ with respect to $L^{\otimes m}$. Note however that it is not always true that
$$  \weight_{0} < \chi/c < \weight_{1}  \,\,\,\,\, \mbox{ iff } \,\,\,\,\,\,   \weight_{0}^{L^{\otimes m}} < m\chi/c < \weight_{1}^{L^{\otimes m}} $$
since in general $ \weight_{1}^{L^{\otimes m}} < m \weight_{1}$ although $ \weight_{0}^{L^{\otimes m}} = m  \weight_{0}$.
\end{rem}

\noindent {\bf Proof of Theorem \ref{fingen}: } 
We will use
induction on the dimension of $U$ to prove a slightly stronger result including

(6) the  tensor power $L^{\otimes c}$ of $L$ induces a very ample line bundle on an inner enveloping quotient $X \inenv U$ for the action of $U$ on $X$ with a $\CC^*$-equivariant embedding
$$X \inenv U \to \PP((H^0(X,L^{\otimes c})^U)^*)$$
as a quasi-projective subvariety, containing the geometric quotient $X^{s,U}/U$ as an open subvariety, with closure $\overline{X \inenv U}$ in  $\PP((H^0(X,L^{\otimes c})^U)^*)$, and

 (7) $X^{s,\hU}_{\min+}$ is a $U$-invariant open subset of $X^{s,U}$ and has a geometric quotient $X^{s,\hU}_{\min+}/U$  which is a $\CC^*$-invariant open subset of $X^{s,U}/U$ and coincides with both the stable and semistable sets $(\overline{X \inenv U})^{s,\CC^*} = (\overline{X \inenv U})^{ss,\CC^*}$ for the $\CC^*$ action with respect to the linearisation on $\calo_{\PP((H^0(X,L^{\otimes c})^U)^*)}(1)$ induced by ${L_\chi^{\otimes c}}$, so that the associated GIT quotient satisfies
$$\overline{X \inenv U}/\!/ \CC^* \cong (X^{s,\hU}_{\min+}/U)/\CC^* \cong X^{s,\hU}_{\min+}/\hU = X\env_{L_\chi^{\otimes c}} \hU.$$ 
When $\dim(U) = 1$ so that $U=\CC^+$, this extended version of Theorem \ref{fingen} including (6) and (7) 
follows immediately from Theorem \ref{fingendimone}.

Now suppose that $\dim(U) > 1$ and that the extended result is true for all strictly smaller values of $\dim(U)$.
We can assume without loss of generality that $U$ is nontrivial.
The centre of
 $U$ is then nontrivial and isomorphic to 
a product of copies of $\CC^+$ on which $\CC^*$ acts with positive weights. So $U$ has a normal
subgroup $U_0$ which is central in $U$ and normal in $\hat{U}$
and is isomorphic to $\CC^+$, such that the given one-parameter group $\CC^* \leq \hU$ of automorphisms
of $U$ preserves $U_0$ and acts on the Lie algebra of $U_0$ with
positive weight. 
By induction on the dimension of $U$, we can now find a subgroup $U_{\dagger}$ of $U$ which is normal in $\hU$ and such that $U/U_{\dagger}$ is one-dimensional and so isomorphic to $\CC^+$, while $\hU/U_{\dagger}$ is a semidirect product of $U/U_{\dagger}$ by $\CC^*$ where $\CC^*$ acts on the Lie algebra of $U/U_{\dagger}$ with strictly positive weight. 
Let $\hU_{\dagger}  = \CC^* \ltimes U_{\dagger}$ be the subgroup of $\hU$ generated by $U_{\dagger}$ and the one-parameter subgroup $\CC^*$. 

By Lemma \ref{lemtensorpower} the linear action of $\hU_{\dagger}$ on $X$ satisfies the condition that semistability is the same as stability. Thus
by induction on the dimension of $U$ we can assume that, for a sufficiently divisible positive integer $c_{\dagger}$,

\noindent (i) the $\hU_{\dagger}$-invariant open subset $X^{s,\hU_{\dagger}}_{\min+}$ of $X$ has a geometric quotient $\pi: X^{s,\hU_{\dagger}}_{\min+} \to
X^{s,\hU_{\dagger}}_{\min+}/\hU_{\dagger}$ by the action of $\hU_{\dagger}$;

\noindent (ii) this geometric quotient $X^{s,\hU_{\dagger}}_{\min+}/\hU_{\dagger}$ is a projective variety and the  tensor power $L^{\otimes c_{\dagger}}$ of $L$ descends to an ample line bundle $L_{(c_{\dagger},\hU_{\dagger})}$ on $X^{s,\hU_{\dagger}}_{\min+}/\hU_{\dagger}$;   

\noindent (iii) the  tensor power $L^{\otimes c_{\dagger}}$ of $L$ induces a very ample line bundle on an inner enveloping quotient $X \inenv U_{\dagger}$ for the action of $U_{\dagger}$ on $X$ with a $\CC^*$-equivariant embedding
$$X \inenv U_{\dagger} \to \PP((H^0(X,L^{\otimes c_{\dagger}})^U_{\dagger})^*)$$
as a quasi-projective subvariety, containing the geometric quotient $X^{s,U_{\dagger}}/U_{\dagger}$ as an open subvariety, with closure $\overline{X \inenv U_{\dagger}}$ in  $\PP((H^0(X,L^{\otimes c_{\dagger}})^{U_{\dagger}})^*)$;

\noindent (iv) $X^{s,\hU_{\dagger}}_{\min+}$ is a $U_{\dagger}$-invariant open subset of $X^{s,U_{\dagger}}$ and has a geometric quotient $X^{s,\hU_{\dagger}}_{\min+}/U_{\dagger}$  which is a $\CC^*$-invariant open subset of $X^{s,U_{\dagger}}/U_{\dagger}$ and, if the rational character $\chi/c_{\dagger}$ is well adapted for the linear action of $\hU_{\dagger}$ with respect to $L$, coincides with both the stable and semistable sets $(\overline{X \inenv U_{\dagger}})^{s,\CC^*} = (\overline{X \inenv U_{\dagger}})^{ss,\CC^*}$ for the $\CC^*$ action with respect to the linearisation induced by ${L_\chi^{\otimes c_{\dagger}}}$ on $\calo_{\PP((H^0(X,L^{\otimes c_{\dagger}})^{U_{\dagger}})^*)}(1)$, so that the associated GIT quotient satisfies
$$\overline{X \inenv U_{\dagger}}/\!/ \CC^* \cong (X^{s,\hU_{\dagger}}_{\min+}/U_{\dagger})/\CC^* \cong X^{s,\hU_{\dagger}}_{\min+}/\hU_{\dagger} = X\env_{L_\chi^{\otimes c_{\dagger}}} \hU_{\dagger}.$$ 

Note that $X^{s,\hU}_{\min +}$ is a $\hU$-invariant open subvariety of $X^{s,\hU_{\dagger}}_{\min +}$. We have an induced linear action of $\hU/U_{\dagger} \cong \CC^+ \rtimes \CC^*$ on $\PP((H^0(X,L^{\otimes c_{\dagger}})^{U_{\dagger}})^*)$ which restricts to a linear action on $\overline{X \inenv U_{\dagger}}$ and to the induced linear action of the open subset $X^{s,\hU}_{\min +}/U_{\dagger}$ of $X^{s,\hU_{\dagger}}_{\min +}/U_{\dagger}
= (\overline{X \inenv U_{\dagger}})^{s,\CC^*} $. We also have
$X^{s,\hU}_{\min +} = \bigcap_{u \in U} X^{s,\CC^*}_{\min +} $ so that
$$X^{s,\hU}_{\min +} /U_{\dagger} = \bigcap_{u \in U/U_{\dagger}} u \, (X^{s,\hU_{\dagger}}_{\min +} /U_{\dagger}) =
\bigcap_{u \in U/U_{\dagger}} u \,\, (\overline{X \inenv U_{\dagger}})^{s,\CC^*}  =  (\overline{X \inenv U_{\dagger}})^{s,\hU/U_{\dagger}} _{\min +}.$$
By Lemma \ref{lemtensorpower} we can apply Theorem \ref{fingendimone} to the action of $\hU/U_{\dagger}$ on the
closure $\overline{X \inenv U_{\dagger}}$ in  $\PP((H^0(X,L^{\otimes c_{\dagger}})^{U_{\dagger}})^*)$ of the inner enveloping quotient $X \inenv U_{\dagger}$ for the action of $U_{\dagger}$ on $X$. It follows that 
$X^{s,\hU}_{\min +} /U_{\dagger} =  (\overline{X \inenv U_{\dagger}})^{s,\hU/U_{\dagger}} _{\min +}$ has a geometric quotient
$$ (X^{s,\hU}_{\min +} /U_{\dagger})/(\hU/U_{\dagger})$$
which is then a geometric quotient for the action of $\hU$ on $X^{s,\hU}_{\min +}$. Furthermore by Theorem \ref{fingendimone} this geometric quotient 
$ (X^{s,\hU}_{\min +} /U_{\dagger})/(\hU/U_{\dagger}) = X^{s,\hU}_{\min +}/\hU$ is 
a projective variety
 and for a sufficiently divisible multiple $c$ of $c_{\dagger}$ the  tensor power $L^{\otimes c}$ of $L$ descends to a very ample line bundle $L_{(c,\hU)}$ on $X^{s,\hU}_{\min+}/\hU$; in addition if $\chi/c =  \weight_{0} + \epsilon$ where $\epsilon >0$ is sufficiently small then 
 $X^{s,\hU}_{\min+}/U_{\dagger} $ is the stable set for the $\hU/U_{\dagger}$-action on $\overline{X \inenv U_{\dagger}}$ with respect to the linearisation induced by $L^{\otimes c}$ and twisted by the rational character $\chi/c$, so 
$$ X^{s,\hU}_{\min+} \subseteq X^{s,\hU}_{L^{\otimes c}_{\chi}}$$
by Remark \ref{rmk:GiSt2}(ii).

Conversely if $\epsilon < \weight_{1} -  \weight_{0}$ then $X^{s,\hU}_{L^{\otimes c}_{\chi}}$ is a $\hU$-invariant subset of $X^{s,\CC^*}_{\min+} = X^{s,\CC^*}_{L^{\otimes c}_{\chi}}$ by Remark \ref{rmk:GiSt2}(i), so
$$ X^{s,\hU}_{L^{\otimes c}_{\chi}} \subseteq \bigcap_{u \in U}  u \, X^{s,\CC^*}_{\min+} = X^{s,\hU}_{\min+}
$$ and hence $ X^{s,\hU}_{L^{\otimes c}_{\chi}}  = X^{s,\hU}_{\min+}
$.

Since the geometric quotient $X^{s,\hU}_{\min +}/\hU = X^{s,\hU}_{L^{\otimes c}_{\chi}}/\hU$ is 
a projective variety
 with  a very  ample line bundle $L_{(c,\hU)}$ induced by the  tensor power $L^{\otimes c}$ of $L$, it follows from 
Proposition \ref{cor:GiFi5} that $X^{s,\hU}_{L^{\otimes c}_{\chi}} = X^{ss,\hU}_{L^{\otimes c}_{\chi}}$, that this geometric quotient coincides with the enveloping quotient $X\env_{L_\chi^{\otimes c}} \hU$, and that if $c$ is replaced with a sufficiently divisible multiple then  the algebra of invariants 
$\bigoplus_{k \geq 0} H^0( X,L_{k\chi}^{\otimes ck})^{\hU}$ is finitely generated and the enveloping quotient $X\env_{L_\chi^{\otimes c}} \hU  \cong \mathrm{Proj}(\oplus_{k \geq 0} H^0( X,L_{k\chi}^{\otimes ck})^{\hU})$ is the associated projective variety.

We can 
apply Theorem \ref{fingendimone} to the action of $\hU/U_{\dagger}$ on the
closure $\overline{X \inenv U_{\dagger}}$ in  $\PP((H^0(X,L^{\otimes c_{\dagger}})^{U_{\dagger}})^*)$ of the inner enveloping quotient $X \inenv U_{\dagger}$ for the action of $U_{\dagger}$ on $X$. It then follows by induction that after replacing $c$ with a sufficiently divisible multiple if necessary, we can assume that there is an inner enveloping quotient $X \inenv U$ for the linear action of $U$ on $X$ with respect to the linearisation $L^{\otimes c}_{\chi}$ obtained by 
considering the induced action of the subgroup $U/U_{\dagger}$ of $\hU/U_{\dagger}$ on $\overline{X \inenv U_{\dagger}}$. We can also assume that the  tensor power $L^{\otimes c}$ of $L$ induces a very ample line bundle on $X \inenv U$ so that there is a $\CC^*$-equivariant embedding
$$X \inenv U \to \PP((H^0(X,L^{\otimes c})^U)^*)$$
of $X \inenv U$ as a quasi-projective subvariety, containing the geometric quotient $X^{s,U}/U$ as an open subvariety, with closure $\overline{X \inenv U}$ in  $\PP((H^0(X,L^{\otimes c})^U)^*)$, such that 
 $X^{s,\hU}_{\min+}$ is a $U$-invariant open subset of $X^{s,U}$ and has a geometric quotient $X^{s,\hU}_{\min+}/U$  which is a $\CC^*$-invariant open subset of $X^{s,U}/U$ and coincides with both the stable and semistable sets $(\overline{X \inenv U})^{s,\CC^*} = (\overline{X \inenv U})^{ss,\CC^*}$ for the $\CC^*$ action with respect to the linearisation on $\calo_{\PP((H^0(X,L^{\otimes c})^U)^*)}(1)$ induced by ${L_\chi^{\otimes c}}$. It then follows that the associated GIT quotient satisfies
$$\overline{X \inenv U}/\!/ \CC^* \cong (X^{s,\hU}_{\min+}/U)/\CC^* \cong X^{s,\hU}_{\min+}/\hU = X\env_{L_\chi^{\otimes c}} \hU$$
and this completes the inductive proof. 
 \hfill $\Box$

\bigskip

We have now proved Theorem \ref{mainthm} and Corollary \ref{cor:invariants2}, which follow immediately from
 Theorem \ref{fingen}. 
Corollary \ref{cor1.2} follows directly as well, since if a
complex linear algebraic group $H$ with unipotent radical $U$ 
acts on a complex algebra $A$ in such a way that the algebra
of $U$-invariants $A^U$ is finitely generated, then there is an induced
action on $A^U$ of the reductive group $R=H/U$, and the algebra of
$H$-invariants
$$A^H = (A^U)^R$$
is finitely generated since $R$ is reductive. 
In the situation of Corollary \ref{cor1.2} when $A$ is the algbra $\oplus_{k \geq 0} H^0( X,L_{k\chi}^{\otimes ck}$ 
then the associated projective variety is the enveloping quotient $X\env H$, and this enveloping quotient is the GIT quotient of the enveloping quotient $X\env\hU$ by the reductive subgroup of $R$ which is its intersection with the kernel of the character $\chi$, with respect to the induced linearisation. The result follows from combining Theorem \ref{fingen} with classical GIT for the action of this reductive subgroup of $R$.

\section{Automorphism groups of toric varieties}

In this section we observe that if $Y$ is a complete simplicial 
toric variety then its automorphism group $Aut(Y)$ satisfies the conditions of
Corollary \ref{cor1.2}, so that every well adapted linear action of
$Aut(Y)$ on a projective variety $X$ with respect to an ample line bundle for which semistability coincides with stability has finitely generated invariants and its enveloping quotient is a geometric quotient of $X^{ss}$.

For this we use the description of $Aut(Y)$ given in \cite{cox}. Let $Y$ be a complete simplicial toric variety over $\CC$ of dimension $n$, and let $S$ be its 
homogeneous coordinate ring in the sense of \cite{cox}. Thus
$$S = \CC[x_\rho : \rho \in \Delta(1)]$$
is a polynomial ring in $d = |\Delta(1)|$ variables $x_\rho$, one for each
one-dimensional cone $\rho$ in the fan $\Delta$ determining the toric variety $Y$. The homogeneous coordinate ring $S$ is graded by setting the degree of a 
monomial $\prod_\rho x_\rho^{a_\rho}$ to be the class of the corresponding Weil
divisor $\sum_\rho a_\rho D_\rho$ in the Chow group $A_{n-1}(Y)$, giving us the decomposition
$$S = \bigoplus_{\alpha \in A_{n-1}(Y)} S_\alpha$$
where $S_\alpha$ is spanned by the monomials of degree $\alpha$. Then we have
$$S_\alpha = S^{'}_{\alpha} \oplus S^{''}_\alpha$$
where $S^{'}_{\alpha}$ is spanned by the $x_\rho$ of degree $\alpha$ and 
$S^{''}_\alpha$ is spanned by the remaining monomials in $S_\alpha$, each being 
a product of at least two variables.

Then by \cite{cox}
Theorem 4.2 and Proposition 4.3, $Aut(Y)$ is an affine algebraic group fitting
into an exact sequence
$$1 \to {\rm Hom}_\ZZ(A_{n-1}(Y),\CC^*) \to \widetilde{Aut}(Y) \to Aut(Y) \to 1$$
with ${\rm Hom}_\ZZ(A_{n-1}(Y),\CC^*)$ isomorphic to a product of a finite group and a torus
$(\CC^*)^{d-n}$, and the identity component $\widetilde{Aut}^0(Y)$ of
$\widetilde{Aut}(Y)$ satisfies
$$\widetilde{Aut}^0(Y) \cong U \rtimes \tilde{R}$$
for 
$$\tilde{R} \cong \prod_\alpha GL(S^{'}_{\alpha})$$
and the unipotent radical $U$ of $\widetilde{Aut}^0(Y)$ is given by
$$U = 1 + \caln$$
where $\caln$ is the ideal
$$\caln = \bigoplus_\alpha {\rm Hom}_\CC(S^{'}_{\alpha}, S^{''}_\alpha)$$
in ${\rm End(S)}$. The reductive group $\tilde{R} \cong \prod_\alpha GL(S^{'}_{\alpha})$ acts in the obvious way on $S$ by identifying $S$ with
the symmetric algebra on $\bigoplus_\alpha S^{'}_{\alpha}$, so that $r \in \tilde{R}$ acts on ${\rm Hom}_\CC(S^{'}_{\alpha},S^{''}_{\alpha}$ for each
$\alpha \in A_{n-1}(Y)$, and thus on $\caln$, via pre-composition with the action of $r$ on $S^{'}_{\alpha}$ and post-composition with the induced action of
$r^{-1}$ on 
$$ S^{''}_{\alpha} \subseteq \bigoplus_{j \geq 2} \Sym^j(\bigoplus_\alpha
S^{'}_{\alpha}).$$
It follows that if we embed $\CC^*$ in $\tilde{R} = \prod_\alpha GL(S^{'}_{\alpha})$ via 
$$t \mapsto (t^{-1}{\rm id}_{S^{'}_{\alpha}})_\alpha$$
where ${\rm id}_{S^{'}_{\alpha}}$ is the identity in $GL(S^{'}_{\alpha})$,
then the weights of the action of $\CC^*$ on the Lie algebra $\caln$ of $U$ are 
all of the form
$$t \mapsto t^{j-1}$$
for some $j \geq 2$, so that $j-1 > 0$. Thus we obtain

\begin{lemma}
If $Y$ is a complete simplicial toric variety then $Aut(Y)$ is
of the form 
$$Aut(Y) \cong U \rtimes R$$
where $U$ is unipotent and $R$ is reductive, and $R$ contains
a one-parameter subgroup $\CC^* \leq R$ such that the action of $\CC^*$ on the Lie algebra of $U$ induced by its conjugation action on $U$ has all weights strictly positive.
\end{lemma}

As an immediate consequence of this lemma and Corollary \ref{cor1.2} we have

\begin{corollary}
Any well adapted linear action of
$H=Aut(Y)$ on a projective variety $X$ with respect to an ample line bundle $L$, for which semistability coincides with stability for the action of its unipotent radical $U$ extended by the central one-parameter subgroup of $Aut(Y)/U$ described above, has finitely generated invariants when $L$ is replaced by a tensor power $L^{\otimes c}$ for a sufficiently divisible positive integer $c$. Furthermore its enveloping quotient $X\env H$ is the associated projective variety and is a categorical quotient of $X^{ss}$ by the action of $H$, while the canonical morphism $\phi:X^{ss} \to X\env H$ is surjective with $\phi(x)=\phi(y)$ if and only if the closures of the $H$-orbits of $x$ and $y$ meet in $X^{ss}$.
\end{corollary}

\section{Jet
differentials and generalised Demailly--Semple jet bundles} \label{jetdifferentials}

Our remaining aim is to apply our results
to a family of examples involving non-reductive
reparametrisation groups which arise in singularity theory and
the study of jets of curves. We borrow
notation from \cite{dem}.

Let $X$ be a complex $n$-dimensional manifold. Green and Griffiths
in \cite{gg} introduced a bundle $J_k \to X$, the bundle
of $k$-jets of germs of parametrised curves in $X$; that is, the
fibre over $x\in X$ is the set of equivalence classes of holomorphic
maps $f:(W,0) \to (X,x)$ where $W$ is an open neighbourhood of $0$ in $\CC$, with the equivalence relation $f\sim g$
if and only if the $j$th derivatives $f^{(j)}(0)=g^{(j)}(0)$ are equal for
$0\le j \le k$. If we choose local holomorphic coordinates
$(z_1,\ldots, z_n)$ on an open neighbourhood $\Omega \subset X$
around $x$, the elements of the fibre $J_{k,x}$ are represented
by the Taylor expansions 
\[f(t)=f(0)+tf'(0)+\frac{t^2}{2!}f''(0)+\ldots +\frac{t^k}{k!}f^{(k)}(0)+O(t^{k+1}) \]

up to order $k$ at $t=0$ of $\CC^n$-valued holomorphic 
maps
\[f=(f_1,f_2,\ldots, f_n):(\CC,0) \to (\CC^n,x).\]
 In these coordinates we have
\[J_{k,x} \cong \left\{(f'(0),\ldots, f^{(k)}(0)/k!)\right\} \cong (\CC^n)^k,\]
which we identify with $\CC^{nk}$.  Note, however, that $J_k$ is not
a vector bundle over $X$, since the transition functions are polynomial, but
not in general linear.

Let $\GG_k$ be the group of $k$-jets of biholomorphisms
\[(\CC,0) \to (\CC,0);\]
that is, the $k$-jets at the origin of local reparametrisations
\[t \mapsto \varphi(t)=\a_1t+\a_2t^2+\ldots +\a_kt^k,\ \ \ \a_1\in
\CC^*,\a_2,\ldots,\a_k \in \CC,\] in which the composition law is
taken modulo terms $t^j$ for $j>k$. This group acts fibrewise on
$J_k$ by substitution. A short computation shows that the action on the fibre
is linear:
\begin{eqnarray*}\label{compose} f \circ
\varphi(t)=f'(0)\cdot(\a_1t+\a_2t^2+\ldots
+\a_kt^k)+\frac{f''(0)}{2!}\cdot (\a_1t+\a_2t^2+\ldots
+\a_kt^k)^2+\ldots \\
\ldots +\frac{f^{(k)}(0)}{k!}\cdot (\a_1t+\a_2t^2+\ldots +\a_kt^k)^k 
\mbox{ (modulo }  t^{k+1})  
\end{eqnarray*}
so the linear action of $\varphi$ on the $k$-jet $(f'(0),\ldots,
f^{(k)}(0)/k!)$ is given by the following matrix multiplication:
\begin{equation}\label{matrixform}
(f'(0),f''(0)/2!,\ldots,f^{(k)}(0)/k!) 
\left(\begin{array}{ccccc}
\a_1 & \a_2 & \a_3 & \cdots  & \a_k \\
0        & \a_1^2 & 2\a_1\a_2 & \cdots &  \a_1\a_{k-1}+\ldots +\a_{k-1}\a_1 \\
0        & 0       & \a_1^3  & \cdots & 3\a_1^2\a_{k-2}+\ldots \\
\cdot    & \cdot   & \cdot    & \cdot &  \cdot \\
0 & 0 & 0 & \cdots  & \a_1^k 
\end{array} \right)
\end{equation}
with $(i,j)$th entry 
\[ 
\sum_{s_1+\ldots +s_i=j}\a_{s_1}\ldots \a_{s_i}\]
for $i,j\le k$. 

There is an exact sequence of groups:
\[0 \rightarrow \UU_k \rightarrow \GG_k \rightarrow \CC^* \rightarrow
0,\] where $\GG_k \to \CC^*$ is the morphism $\varphi \to
\varphi'(0)=\a_1$ in the notation used above, and
\[\GG_k=\UU_k \ltimes \CC^*\]
is a semi-direct product of $\UU_k$ by $\CC^*$. With the above identification, $\CC^*$ is
the subgroup of diagonal matrices satisfying $\a_2=\ldots =\a_k=0$ and
$\UU_k$ is the unipotent radical of $\GG_k$, i.e. the subgroup of matrices
with $\a_1=1$. The action
of  $\l \in \CC^*$ on $k$-jets is described by
\[\l\cdot (f'(0),f''(0),\ldots ,f^{(k)}(0))=(\l f'(0),\l^2 f''(0),\ldots,
\l ^kf^{(k)}(0)).\]

Let $\cale_{k,m}^n$ denote the vector space of polynomials
$Q(u_1,u_2,\ldots, u_k)$,  of weighted degree $m$, with respect to
this $\CC^*$ action, where $u_i=f^{(i)}(0)$; that is, such that
\[Q(\l u_1,\l^2 u_2,\ldots, \l^k u_k)=\l^m Q(u_1,u_2,\ldots,
u_k).\]
Elements of $\cale_{k,m}^n$ have the form
\[Q(u_1,u_2,\ldots, u_k)=\sum_{|\a_1|+2|\a_2|+\ldots
+k|\a_k|=m}u_1^{\a_1}u_2^{\a_2}\ldots u_k^{\a_k},\] where
$\a_1,\a_2,\ldots, \a_k$ are multi-indices of length $n$.

$\cale_{k,m}^n$ can be identified with
 the fibre of the vector bundle $E^{GG}_{k,m} \to
X$  introduced by Green and Griffiths in
\cite{gg}, whose fibres consist of polynomials on the fibres
of $J_k$ of weighted degree $m$ with respect to the fibrewise
$\CC^*$ action on $J_k$.

The action of $\GG_k$ naturally induces an action on the vector space
$$\cale_k^n = \bigoplus_{m\geq 0} \cale_{k,m}^n = \calo(J_{k,x})$$ 
of polynomial functions on $J_{k,x}$.
Following Demailly (\cite{dem}), we define 
$\tilde{\cale}_{k,m}^n \subset \cale_{k,m}^n$ 
to be the vector space of
$\UU_k$-invariant polynomials of weighted degree $m$, i.e. those which satisfy
\[Q((f\circ \varphi)',(f\circ \varphi)'', \ldots, (f\circ
\varphi)^{(k)})=\varphi'(0)^m\cdot Q(f',f'',\ldots, f^{(k)}).\]
Thus $\tilde{\cale}_{k}^n = \bigoplus_{m\geq 0} \tilde{\cale}_{k,m}^n
= \calo(J_{k,x})^{\UU_k}$ consists of the polynomials functions on $J_{k,x}$ which
are invariant under the induced action of $\UU_k$ on $\calo(J_{k,x})$.
The corresponding bundle of invariants is the Demailly-Semple bundle
of algebras $E_k^n=\oplus_m E^n_{k,m} \subset
\oplus_m E_{k,m}^{GG}$ with fibres $\tilde{\cale}_{k}^n = \bigoplus_{m\geq 0} \tilde{\cale}_{k,m}^n
= \calo(J_{k,x})^{\UU_k}$.

This bundle of graded algebras  $E_k^n=\oplus_m E_{k,m}^n$  has 
been an important object of study for a long time. The invariant jet
differentials
play a crucial role in the strategy developed by Green, Griffiths,
Bloch, Ahlfors, Demailly, Siu and others to prove Kobayashi's 1970 hyperbolicity
conjecture  \cite{ahl,blo,dem,dmr,dr,gg, kob,merker, siu1,siu2,siu3}.

We can now apply 
Theorem \ref{mainthm}, Corollary \ref{cor:invariants} and Remark \ref{cor:invariants2} to  
 linear
action of $\GG_k$  on the projective variety associated to $J_{k,x}$. In this case we can also apply the results of \cite{BK15} since $\GG_k$ is a subgroup of $\GL(k;\CC)$ which is \lq generated along the first row' in the sense of \cite{BK15}, and the action of $\GG_k$ extends to $\GL(k;\CC)$.

We can also consider a generalised version of the Demailly-Semple jet differentials to which the results of \cite{BK15} do not apply. Instead of germs of holomorphic maps $\CC \to X$, we now consider higher
dimensional holomorphic objects in $X$, and therefore we fix a
parameter $1\le p \le n$, and study germs of holomorphic maps $\CC^p \to X$.

Again we fix the degree $k$ of these maps, and introduce the bundle
$J_{k,p} \to X$ of $k$-jets of germs of holomorphic maps $\CC^p \to X$. 
With respect to local holomorphic coordinates near $x\in X$ the
fibre over $x$ is identified with the set of equivalence classes of holomorphic
maps $f:(\CC^p,0) \to (\CC^n,x)$, with the equivalence relation $f\sim
g$ if and only if all derivatives $f^{(j)}(0)=g^{(j)}(0)$ are equal for
$0\le j \le k$. Equivalently the elements
of the fibre $J_{k,p,x}$ are the Taylor expansions 
\[f(\bu)=x+\bu f'(0)+\frac{\bu^2}{2!}f''(0)+\ldots +\frac{\bu^k}{k!}f^{(k)}(0)+O(|\bu|^{k+1}) \]
around
$\bu=0$ up to order $k$
of $\CC^n$-valued maps
\[f=(f_1,f_2,\ldots, f_n):(\CC^p,0) \to (\CC^n,x).\]
Here
\[f^{(i)}(0)\in \Hom(\Sym^i\CC^p,\CC^n)\]
so that in these coordinates the fibre is
\[J_{k,p,x}=\left\{(f'(0),\ldots, f^{(k)}(0)/k!)\right\}=\CC^{n( {k+p \choose k}-1)}\]
which is a finite dimensional vector space.

Let $\GG_{k,p}$ be the group of $k$-jets of germs of biholomorphisms
of $(\CC^p,0)$, that is, the group of biholomorphic maps
\begin{equation} \bu \to \varphi(\bu)=\Phi_1\bu+\Phi_2\bu^2+\ldots +\Phi_k\bu^k=
\sum_{1 \leq i_1 + \cdots +i_p \leq k}a_{i_1\ldots
i_p}u_1^{i_1}\ldots u_p^{i_p}
\end{equation}
for which $\Phi_i \in \Hom(\Sym^i\CC^p,\CC^p)$
and $\Phi_1 \in \Hom(\CC^p,\CC^p)$ is non-degenerate.
Then $\GG_{k,p}$ admits a natural fibrewise right action on $J_{k,p}$,
which consist of reparametrizing the $k$-jets of holomorphic
$p$-discs. A similar computation to at \eqref{compose} shows that
\[f \circ\varphi(\bu)=(f'(0)\Phi_1)\bu+(f'(0)\Phi_2+\frac{f''(0)}{2!}\Phi_1^2)\bu^2+\ldots
+ \ldots +\sum_{i_1+\ldots +i_l=k}
(\frac{f^{(l)}(0)}{l!}\Phi_{i_1}\ldots \Phi_{i_l})\bu^l.\]
This is a linear action on the fibres $J_{k,p,x}$ with matrix
given by
\begin{equation} \label{(4)}
\left(
\begin{array}{ccccc}
\Phi_1 & \Phi_2 & \Phi_3 & \ldots & \Phi_k \\
0 & \Phi_1^2 & \Phi_1\Phi_2 & \ldots &       \\
0 & 0 & \Phi_1^3 & \ldots & \\
. & . & . & . & . \\
& & & & \Phi_1^k
\end{array}
\right),
\end{equation}
where
$\Phi_i$ is a $p \times \dim (\Sym^i \CC^p)$-matrix, the $i$th
degree component of the map $\Phi$ and 
the $p \times p$matrix $\Phi_1$ is invertible. Here
$\Phi_{i_1} \ldots \Phi_{i_l}$ is the matrix of the map $\Sym^{i_1+\ldots
+i_l}(\CC^p)\to \Sym^{l}\CC^p$, which is induced by 
\[ 
\Phi_{i_1} \otimes \cdots \otimes
\Phi_{i_l}:(\CC^p)^{\otimes i_1} \otimes \cdots \otimes
(\CC^p)^{\otimes i_l} \to (\CC^p)^{\otimes l}\]

The linear group $\GG_{k,p}$ is generated along its first $p$ rows,
in the sense that the parameters in the first $p$ rows are independent, and
all the remaining entries are polynomials in these parameters. The
only condition which the parameters must satisfy is that the determinant of the first
diagonal $p\times p$ block is nonzero.
Note that $\GG_{k,p}$ is an extension of its unipotent radical $\UU_{k,p}$
(given by $\Phi_1=1$ by
$GL(p;\CC)$ (given by $\Phi_i = 0$ for $i>1$), so we have an exact sequence
\[0 \rightarrow \UU_{k,p} \rightarrow \GG_{k,p} \rightarrow GL(p;\CC) \rightarrow
0.\] 
The central $\CC^*$ of $GL(p;\CC)$ corresponds to the diagonal matrices
with entries $t,t^2, \ldots, t^k$ for $t \in \CC^*$ where $t^i$ occurs
$\dim(\Sym^i(\CC^p))$ times, and these act by conjugation on the Lie algebra
of $\UU_{k,p}$ with weights $i-1$ for $2 \leq i \leq k$. Thus 
by Corollary \ref{cor1.2} we have

\begin{corollary} \label{po}
Any linear action of $\GG_{k,p}$ 
 on a projective variety $X$ with respect to an ample line bundle $L$ for which semistability coincides with stability for the action of $\UU_{k,p}$ extended by the central one-parameter subgroup of $\GL(p;\CC)$ has finitely generated invariants when $L$ is replaced by a tensor power $L^{\otimes c}$ for a sufficiently divisible positive integer $c$ and the linearisation is twisted by a well adapted rational character. Furthermore its enveloping quotient $X\env \GG_{k,p}$ is the associated projective variety and is a categorical quotient of $X^{ss}$ by the action of $\GG_{k,p}$, while the canonical morphism $\phi:X^{ss} \to X\env \GG_{k,p}$ is surjective with $\phi(x)=\phi(y)$ if and only if the closures of the $\GG_{k,p}$-orbits of $x$ and $y$ meet in $X^{ss}$.
\end{corollary}

\begin{definition}
The generalized Demailly-Semple jet bundle $E_{k,p,m} \to X$ of
invariant jet differentials of
order $k$ and weighted degree $(m,\ldots, m)$
has fibre at $x \in X$ consisting of
complex-valued polynomials
$Q(f'(0),f''(0)/2,\ldots ,f^{(k)}(0)/k!)$ on the fibre $J_{k,p,x}$ of $J_{k,p}$, 
which transform under any reparametrization $\phi\in \GG_{k,p}$ of
$(\CC^p,0)$ as
\[Q(f \circ \phi)=(J_{\phi}(0))^mQ(f)\circ \phi,\]
where $J_\phi(0)$ denotes the Jacobian at 0 of $\phi$;
that is, $J_\phi(0)=\det \Phi_1$ when $\phi$ is given as at
(\ref{(4)}). Thus the
generalized Demailly-Semple bundle $E_{k,p}=\oplus
E_{k,p,m}$ of invariant jet differentials of order $k$ has fibre at $x \in X$
given by the generalized Demailly-Semple algebra
$\calo(J_{k,p,x})^{\UU_{k,p} \rtimes SL(p;\CC)}$. 
\end{definition}

We can apply Corollary \ref{po} to the linear action of $\GG_{k,p}$ on the projective space 
$X=\PP(J_{k,p,x})$ with respect to the line bundle $L=\calo_{\PP(J_{k,p,x})}(1)$ satisfying
$$ \calo(J_{k,p,x}) = \oplus_{j \geq 0} H^0( X,L^{\otimes j}).$$
As at Remark 
\ref{cor:invariants2}, by considering a diagonal action on $X \times \PP^1$,
we can deduce that the algebra  $\oplus_{m=0}^\infty H^0(X\times \PP^1,L_{m\chi}^{\otimes cm} \otimes \mathcal{O}_{\PP^1}(M))^{\GG_{k,p}}$ of $\GG_{k,p}$-invariants on $X \times \PP^1$  is 
finitely generated  when $M>>1$ and $c$ is a sufficiently divisible positive integer and the linear action has been twisted by a suitable  rational character $\chi/c$. This finitely generated graded algebra can be identified with the subalgebra of the 
generalized Demailly-Semple algebra
$\calo(J_{k,p,x})^{\UU_{k,p} \rtimes SL(p;\CC)}$
 generated by the $\UU_{k,p} \rtimes SL(p;\CC)$-invariants in $\oplus_{m=0}^\infty H^0(X,L^{\otimes cm})^{{\UU_{k,p} \rtimes SL(p;\CC)}}$ which are weight vectors with non-negative weights for the action of the central one-parameter subgroup of $GL_p$ after twisting by a suitable character $\chi$. This twisting is such that the matrix (\ref{(4)}) is replaced with its multiple by $(\det(\Phi_1)^{-(1/p) - \epsilon}$ for $0<\epsilon << 1$, so the only weight vectors $\sigma \in H^0(X,L) = \bigoplus_{i=1}^k \sym^i(\CC^p)$ with non-negative weights are the sections $\sigma$ in $\sym^1(\CC^p) = \CC^p$, which have weight $p\epsilon$. It therefore follows that the localisation 
$\calo(J_{k,p,x})^{\UU_{k,p} \rtimes SL(p;\CC)}_\sigma$ of the generalized Demailly-Semple algebra
$\calo(J_{k,p,x})^{\UU_{k,p} \rtimes SL(p;\CC)}$ at any such $\sigma$ is finitely generated (cf. \cite{dmr,merker}).

\end{document}